\documentclass{article}%
\usepackage{amsmath}
\usepackage{amsfonts}
\usepackage{amssymb}
\usepackage{graphicx}%
\providecommand{\U}[1]{\protect\rule{.1in}{.1in}}
\newtheorem{theorem}{Theorem}

\newtheorem{corollary}[theorem]{Corollary}

\newtheorem{definition}[theorem]{Definition}

\newtheorem{lemma}[theorem]{Lemma}

\newtheorem{proposition}[theorem]{Proposition}
\newtheorem{remark}[theorem]{Remark}

\newenvironment{proof}[1][Proof]{\noindent\textbf{#1.} }{\ \rule{0.5em}{0.5em}}
\begin{document}

\title{About the structure of attractors for a nonlocal Chafee-Infante problem}
\author{R. Caballero$^{1}$, A.N. Carvalho$^{2}$, P. Mar\'{\i}n-Rubio$^{3}$ and
Jos\'{e} Valero$^{1}$\\$^{1}${\small Centro de Investigaci\'{o}n Operativa, Universidad Miguel
Hern\'{a}ndez de Elche,}\\{\small Avda. Universidad s/n, 03202, Elche (Alicante), Spain}\\{\small E.mail:\ jvalero@gmail.com}\\$^{2}${\small Instituto de Ci\^{e}ncias Matem\'{a}ticas e de Computa\c{c}ao,
Universidade de S\~{a}o Paulo,}\\{\small Campus de S\~{a}o Carlos, Caixa Postal 668, 13560-970 S\~{a}o Carlos,
SP, Brazil}\\{\small Email: andcarva@icmc.usp.br}\\$^{3}${\small Dpto. Ecuaciones Diferenciales y An\'{a}lisis Num\'{e}rico,}\\{\small Universidad de Sevilla, C/Tarfia, 41012-Sevilla, Spain}\\{\small E.mail: pmr@us.es}}
\date{}
\maketitle

\begin{abstract}
In this paper, we study the structure of the global attractor for the
multivalued semiflow generated by a nonlocal reaction-diffusion equation in
which we cannot guarantee uniqueness of the Cauchy problem.

First, we analyse the existence and properties of stationary points, showing
that the problem undergoes the same cascade of bifurcations as in the
Chafee-Infante equation. Second, we study the stability of the fixed points
and establish that the semiflow is dynamically gradient. We prove that the
attractor consists of the stationary points and their heteroclinic connections
and analyse some of the possible connections.

\end{abstract}

\bigskip



\section{Introduction}

Ordinary and partial differential equations play a key role in modeling for
all sciences: Engineering, Physics, Chemistry, Biology, Medicine, Economy and
many others. The right understanding of the behavior of solutions (in
particular, well-posedness versus blow-up) means not only to predict the
future of trajectories but also to establish strategies for control (i.e.
optimization). Concerning PDE and Economy, it is interesting to cite the nice
survey \cite{bcm14} and the references therein on many different problems
dealing with effects as aggregation and repulsion, optimal control, mean-field
games, and so on as applications.

Parabolic PDE models reflect the diffusion phenomena due to \emph{local
touching of molecules} and dissipation of energy, and when different internal
and external factors come into play, it links naturally to some
reaction-diffusion models, as the growth versus capacity of the environment in
Biology or the endogenous growth versus the neoclassical theories in economy.
In particular, capital accumulation distribution in space and time following
spatial extensions of the continuous Ramsey model \cite{ramsey} by Brito
\cite{brito01,brito04,brito12} and others later uses the semilinear parabolic
PDE
\[
\partial_{t}u-\alpha\Delta u=f(u)-c.
\]
This spatiality introduces important issues about the steady states
distribution as well as the dynamic evolution, convergence, local interaction
among local agents, and so on.

Not for the sake of generality but for real modeling purposes, in the last two
decades the increment of nonlocal PDE models that attempt to capture in a more
accurate way the real spreading of the problem (density of population, capital
accumulation, consumption or prices and innovation indexes, and so on) has
been very important. Firstly we might comment about extensions by using some
nonlocal operators acting in the right-hand side of the PDE and/or the
boundary conditions as integral operators, leading to integro-differential
equations. Among others we can cite \cite{anita} for a system coupling capital
and pollution stock model, a population dynamic model in \cite{deng-wu}
\[
\partial_{t}u-\alpha\Delta u=u\left(  f(u)-\alpha\int_{\mathbb{R}^{N}%
}g(x-y)u(y,t)dy\right)  ,
\]
the elliptic (stationary) counterpart in population/physics models as the
Fischer-KPP \cite{AK}, or a logistic model \cite{delgado}. Secondly, we wish
to point out that the nonlocal extensions have also been performed on the
diffusion operators as well. The literature about fractional laplacian is vast
nowadays. However, let us concentrate in an intermediate step. Coming
originally from modeling of bacteria population in Biology, the introduction
of a nonlocal viscosity in front of the laplacian has become an interesting
problem for different applications and for its mathematical study, as for
example occurs in the equation
\[
u_{t}-a(\int_{\Omega}g(y)u(t,y)dy)\Delta u=f(t).
\]
In this way, the spreading (or aggregating/concentrating) effects are given by
the increasing (resp. non-increasing) function $a$ as a viscosity nonlocal
coefficient. One should cite Prof. Chipot and his collaborators
\cite{ChipotLovat97,d3,d41,d5,d4,ChipotValente,ZhengChipot} among others for a
detailed analysis including existence, uniqueness, steady states and
convergence of evolutionary solutions to equilibria.

When the reaction term $f$ depends on the unknown $u$
\begin{equation}
u_{t}-a(\Phi_{\Omega}(u(t))\Delta u=f(t,u) \label{Nonlocal}%
\end{equation}
(here the functional $\Phi_{\Omega}$ may represent a general nonlocal
functional acting over the whole domain $\Omega,$ for instance $\Vert
u(t)\Vert_{H_{0}^{1}}^{2}$ or $\int_{\Omega}g(y)u(t,y)dy$) equilibria are
difficult to analyse. Opposite to ordinary differential equations, the
analysis of existence of stationary states for the above problem is much more
involved. Also, comparing with reaction-diffusion equations with local
diffusion, another difficulty is that in general a Lyapunov functional is not
known to exist in most cases.

The dynamical analysis of problem (\ref{Nonlocal}) and in particular the
existence of global attractors have been established till now in several
papers (cf. \cite{Ahn,CabMarVal,CaHeMa15,d0,CaHeMa18}). Other differential
operators as the $p$-laplacian coupled with nonlocal viscosity has also been
considered (cf. \cite{CAHeMa17,CaHeMa18,CaHeMa18B}). However, in general
little is known about the internal structure of the attractor, which is very
important as it gives us a deep insight into the long-term dynamics of the
problem. When we manage to obtain a Lyapunov functional some insights can be obtained.

If we consider the non-local equation%
\begin{equation}
\dfrac{\partial u}{\partial t}-a(\Vert u\Vert_{H_{0}^{1}}^{2})\dfrac
{\partial^{2}u}{\partial x^{2}}=\lambda f(u) \label{Nonlocal2}%
\end{equation}
with Dirichlet boundary conditions, then it is possible to define a suitable
Lypaunov functional. In \cite{CabMarVal} it is shown that regular and strong
solutions generate (possibly) multivalued semiflows having a global attractor
which is described by the unstable set of the stationary points. Although this
is already a good piece of information, our goal is to describe the structure
of the attractor as accurately as possible. For this aim we need to study the
particular situation where the domain is one-dimensional and the function $f$
is of the type of the standard Chafee-Infante problem, for which the dynamics
inside the attractor has been completely understood \cite{Henry85}.

The first step when studying the structure of the attractor consists in
analysing the stationary points. In the case where the function $f$ is odd and
equation (\ref{Nonlocal2}) generates a continuous semigroup the existence of
fixed points of the type given in the Chafee-Infante problem was established
in \cite{CarLiLuMo}. Moreover, if $a$ is non-decreasing, then they coincide
with the ones in the Chafee-Infante problem and, moreover, in
\cite{carvalestef} the stability and hyperbolicity of the fixed points is
studied. In this paper we extend these results for a more general function $f$
(not necessarily odd and for which we do not known whether the Cauchy problem
has a unique solution or not), showing that equation (\ref{Nonlocal2})
undergoes the same cascade of bifurcations as the Chafee-Infante equation.
Moreover, when we allow the function $a$ to decrease, though the problem
possesses at least the same fixed point as in the Chafee-Infante problem, we
show that more equilibria can appear. For a non-decreasing function $a$ and an
odd function $f$ we prove also that even when uniqueness fails the stability
of the fixed points is the same as for the corresponding ones in the
Chafee-Infante problem. Finally, we are able to prove that in this last case
the semiflow is dynamically gradient with respect to the disjoint family of
isolated weakly invariant sets generated by the equilibria, which is ordered
by the number of zeros of the fixed points. More precisely, the attractor
consists of the set of equilibria and their heteroclinic connections and a
connection from a fixed point to another is allowed only if the number of
zeros of the first one is greater.

In Section 3 we study the existence of strong solutions of the Cauchy problem
in the space $H_{0}^{1}$. In Section 4 we prove that strong solutions generate
a multivalued semiflow in $H_{0}^{1}$ having a global attractor which is equal
to the unstable set of the stationary points. In Section 5 we study the
existence and properties of equilibria. In Section 6 we analyse the stability
of the fixed points and establish that the semiflow is dynamically gradient.

\section{Setting of the problem}

Let us consider the following problem
\begin{equation}
\left\{
\begin{array}
[c]{l}%
\dfrac{\partial u}{\partial t}-a(\Vert u\Vert_{H_{0}^{1}}^{2})\dfrac
{\partial^{2}u}{\partial x^{2}}=\lambda f(u)+h(t),\quad t>0,x\in\Omega,\\
u(t,0)=u(t,1)=0,\\
u(0,x)=u_{0}(x),
\end{array}
\right.  \label{problem1}%
\end{equation}
where $\Omega=(0,1)$ and $\lambda>0$. Throughout the paper we will use the
following conditions (but not all of them at the same time):

\begin{enumerate}
\item[(A1)] $f\in C(\mathbb{R})$.

\item[(A2)] $f(0)=0.$

\item[(A3)] $f^{\prime}(0)$ exists and $f^{\prime}(0)=1$.

\item[(A4)] $f$ is strictly concave if $u>0$ and strictly convex if $u<0$.

\item[(A5)] Growth and dissipation conditions: for $p\geq2,\ C_{i}%
>0,\ i=1,..,4$, we have%
\begin{equation}
|f(u)|\leq C_{1}+C_{2}|u|^{p-1}, \label{A51}%
\end{equation}
\begin{equation}
f(u)u\leq C_{3}-C_{4}|u|^{p}\text{, if }p>2, \label{A52}%
\end{equation}
\begin{equation}
\limsup_{u\rightarrow\pm\infty}\frac{f(u)}{u}\leq0\text{, if }p=2. \label{A53}%
\end{equation}

\item[(A6)] The function $a\in C(\mathbb{R}^{+})$ satisfies:
\[
a(s)\geq m>0.
\]

\item[(A7)] The function $a\in C(\mathbb{R}^{+})$ satisfies:%
\[
a(s)\leq M_{1},\quad\forall s\geq0,
\]
where $M_{1}>0.$

\item[(A8)] The function $a\in C(\mathbb{R}^{+})$ is non-decreasing.

\item[(A9)] $h\in L_{loc}^{2}\left(  0,+\infty;L^{2}\left(  \Omega\right)
\right)  .$

\item[(A10)] $h$ does not depend on time and $h\in L^{2}\left(  \Omega\right)
.$
\end{enumerate}

We define the function $\mathcal{F}(u)=\int_{0}^{u}f(s)ds$. We observe that
from (\ref{A51}) we have%
\begin{equation}
|\mathcal{F}(s)|\leq\widetilde{C}(1+|s|^{p})\quad\forall s\in\mathbb{R}%
\text{,} \label{condFA}%
\end{equation}
whereas (\ref{A52}) implies%
\begin{equation}
\mathcal{F}(s)\leq\widetilde{\kappa}-\widetilde{\alpha}_{1}|s|^{p}.
\label{condFB}%
\end{equation}
Also, from condition (\ref{A53}) it follows that for all $\varepsilon>0$,
there exists a constant $M>0$ such that $\frac{f(u)}{u}\leq\varepsilon,$ for
all $|u|\geq M$. Hence, there exists $m_{\varepsilon}>0$ such that
\begin{equation}
f(u)u\leq m_{\varepsilon}+\varepsilon u^{2},\quad\forall u\in\mathbb{R}.
\label{condf}%
\end{equation}
In addition, it follows that
\begin{equation}
\mathcal{F}(u)\leq\varepsilon u^{2}+C_{\varepsilon}, \label{condF}%
\end{equation}
where $C_{\varepsilon}>0$. These two inequaities are also true under condition
(\ref{A52}).

The main aim of this paper consists in describing in as much detail as
possible the internal structure of the global attractor in a similar way as
for the classical Chafee-Infante equation.

Some of these conditions will be used all the time, whereas other ones will be
used only in certain results. In particular, the function $h$ will be
considered as a time-dependent function satisfying (A9) only for establishing
the existence of solution for problem (\ref{problem1}). However, since we will
study the asymptotic behaviour of solutions in the autonomous situation, for
the second part concerning the existence and properties of global attractors
the function $h$ will be time-independent, so assumption (A10) will be used
instead. Finally, in order to study the structure of the global attractors in
terms of the stationary points and their possible heteroclinic connections we
will assume that $h\equiv0.$

Throughout the paper, $\left\Vert \text{\textperiodcentered}\right\Vert _{X}$
will denote the norm in the Banach space $X.$

\section{Existence of solutions}

In this section we will establish the existence of strong solutions for
problem (\ref{problem1}) with initial condition in the phase space $H_{0}%
^{1}\left(  \Omega\right)  $. Although we will follow the same lines of a
similar result given in \cite{CabMarVal}, we would like to point out that in
the present case, as we are working in a one-dimensional problem, the
assumptions on the function $f$ are much weaker. In particular, we do not need
to impose a growth assumption of any kind.

\begin{definition}
For $u_{0}\in L^{2}(\Omega)$, a weak solution to (\ref{problem1}) is an
element $u\in L^{\infty}(0,T;L^{2}(\Omega))\cap L^{2}(0,T;H_{0}^{1}(\Omega)),$
for any $T>0,$ such that
\begin{equation}
\frac{d}{dt}(u,v)+a(\Vert u\Vert_{H_{0}^{1}}^{2})(\nabla u,\nabla
v)=\lambda(f(u),v)+(h(t),v)\quad\forall v\in H_{0}^{1}(\Omega
),\label{equationweak}%
\end{equation}
where the equation is understood in the sense of distributions.
\end{definition}

As usual, let $A:D(A)\rightarrow H,\ D(A)=H^{2}(\Omega)\cap H_{0}^{1}\left(
\Omega\right)  ,$ be the operator $A=-\dfrac{d^{2}}{dx^{2}}$ with Dirichlet
boundary conditions. This operator is the generator of a $C_{0}$-semigroup
$T(t)=e^{-At}$.

\begin{definition}
For $u_{0}\in H_{0}^{1}(\Omega)$, a strong solution to (\ref{problem1}) is a
weak solution with the extra regularity $u\in L^{\infty}(0,T;H_{0}^{1}%
(\Omega))$, $u\in L^{2}(0,T;D(A))$ and $\dfrac{du}{dt}\in L^{2}(0,T;L^{2}%
(\Omega))$ for any $T>0.$
\end{definition}

\begin{remark}
We observe that if $u$ is a strong solution, then $u\in C([0,T];H_{0}%
^{1}(\Omega))$ (see \cite[p.102]{sellyou}). By this way, the initial condition
makes sense.
\end{remark}

\begin{remark}
Since $\dfrac{du}{dt}\in L^{2}\left(  0,T;L^{2}\left(  \Omega\right)  \right)
$ for any strong solution, in this case equality (\ref{equationweak}) is
equivalent to the following one:%
\begin{align}
&  \int_{0}^{T}\int_{\Omega}\frac{du\left(  t,x\right)  }{dt}\xi\left(
t,x\right)  dxdt-\int_{0}^{T}a(\Vert u(t)\Vert_{H_{0}^{1}}^{2})\int_{\Omega
}\frac{\partial^{2}u}{\partial x^{2}}\xi dxdt\label{EquationRegular}\\
&  =\int_{0}^{T}\int_{\Omega}\lambda f\left(  u\left(  t,x\right)  \right)
\xi\left(  t,x\right)  dxdt+\int_{0}^{T}\int_{\Omega}h\left(  t,x\right)
{\xi}\left(  t,x\right)  dxdt,\nonumber
\end{align}
for all $\xi\in L^{2}\left(  0,T;L^{2}\left(  \Omega\right)  \right)  .$
\end{remark}

\begin{theorem}
\label{existence}Assume conditions (A1), (A6) and (A9). Assume also the
existence of constants $\beta,\gamma>0$ such that%
\begin{equation}
f\left(  u\right)  u\leq\gamma+\beta u^{2}\text{ for all }u\in\mathbb{R}.
\label{Diss}%
\end{equation}
Then, for any $u_{0}\in H_{0}^{1}(\Omega)$ problem (\ref{problem1}) has at
least one strong solution.
\end{theorem}

\begin{remark}
Assumption (\ref{Diss}) is weaker than the dissipative property (\ref{condf})
as the constant $\varepsilon$ is arbitrarily small. Due to the fact that we
are working in a one-dimensional domain, no growth condition of the type given
in (A5) is necessary in order to prove existence of solutions. Also,
(\ref{Diss}) implies that%
\begin{equation}
F\left(  u\right)  \leq\widetilde{\gamma}+\widetilde{\beta}u^{2} \label{Diss2}%
\end{equation}
for some constants $\widetilde{\gamma},\widetilde{\beta}>0$.
\end{remark}

\begin{proof}
Consider a fixed value $T>0.$ In order to use the Faedo-Galerkin method let
$\{w_{j}\}_{j\geq1}$ be the sequence of eigenfunctions of $-\Delta$ in
$H_{0}^{1}(\Omega)$ with homogeneous Dirichlet boundary conditions, which
forms a special basis of $L^{2}(\Omega)$. Since $\Omega$ is a bounded regular
domain, it is known that $\{w_{j}\}\subset H_{0}^{1}(\Omega)$ and that
$\cup_{n\in\mathbb{N}}V_{n}$ is dense in the spaces $L^{2}(\Omega)$ and
$H_{0}^{1}(\Omega)$, where $V_{n}=span[w_{1},\ldots,w_{n}].$ As usual, $P_{n}$
will be the orthogonal projection in $L^{2}\left(  \Omega\right)  $, that is
\[
z_{n}:=P_{n}z=\sum_{j=1}^{n}(z,w_{j})w_{j},
\]
and $\lambda_{j}$ will be the eigenvalues associated to the eigenfunctions
$w_{j}$. For each integer $n\geq1$, we consider the Galerkin approximations
\[
u_{n}(t)=\sum_{j=1}^{n}\gamma_{nj}(t)w_{j},
\]
which are given by the following nonlinear ODE system
\begin{equation}
\left\{
\begin{array}
[c]{l}%
\dfrac{d}{dt}(u_{n},w_{i})+a(\Vert u_{n}\Vert_{H_{0}^{1}}^{2})(\nabla
u_{n},\nabla w_{i})=\lambda(f(u_{n}),w_{i})+(h,w_{i})\quad\forall
i=1,\ldots,n,\\
u_{n}(0)=P_{n}u_{0}.
\end{array}
\right.  \label{1.6}%
\end{equation}
We observe that $P_{n}u_{0}\rightarrow u_{0}$ in $H_{0}^{1}(\Omega)$. This
Cauchy problem possesses a solution on some interval $[0,t_{n})$ and by the
estimates in the space $L^{2}(\Omega)$ of the sequence $\{u_{n}\}$ given below
for any $T>0$ such a solution can be extended to the whole interval $[0,T]$.

Firstly, multiplying the equation in (\ref{1.6}) by $\gamma_{ni}(t)$ and
summing from $i=1$ to $n$, we obtain
\begin{equation}
\frac{1}{2}\frac{d}{dt}\Vert u_{n}(t)\Vert_{L^{2}}^{2}+a(\Vert u_{n}%
\Vert_{H_{0}^{1}}^{2})\Vert u_{n}(t)\Vert_{H_{0}^{1}}^{2}=\lambda
(f(u_{n}(t),u_{n}(t))+(h(t),u_{n}(t))\quad\text{for\ a.e.}\ t\in
(0,t_{n}).\label{7}%
\end{equation}
Using the Young and Poincar\'{e} inequalities we deduce that
\[
(h(t),u_{n}(t))\leq\frac{m}{2}\Vert u_{n}(t)\Vert_{H_{0}^{1}}^{2}+\frac
{1}{2\lambda_{1}m}\Vert h(t)\Vert_{L^{2}}^{2},
\]
where $m$ is the constant from (A6). Hence, from (A6), (\ref{Diss}) and
(\ref{7}) it follows that%
\[
\frac{1}{2}\frac{d}{dt}\Vert u_{n}(t)\Vert_{L^{2}}^{2}+\frac{m}{2}\Vert
u_{n}(t)\Vert_{H_{0}^{1}}^{2}\leq\lambda\gamma|\Omega|+\beta\lambda\Vert
u_{n}(t)\Vert_{L^{2}}^{2}+\frac{1}{2\lambda_{1}m}\Vert h(t)\Vert_{L^{2}}^{2}.
\]
We infer that%
\begin{align}
\Vert u_{n}(t)\Vert_{L^{2}}^{2}  & \leq\left\Vert u_{n}\left(  0\right)
\right\Vert _{L^{2}}^{2}e^{2\beta\lambda t}+\int_{0}^{t}e^{2\beta
\lambda\left(  t-s\right)  }\left(  2\lambda\gamma|\Omega|+\frac{1}%
{\lambda_{1}m}\Vert h(s)\Vert_{L^{2}}^{2}\right)  ds\label{9}\\
& \leq\left\Vert u_{n}\left(  0\right)  \right\Vert _{L^{2}}^{2}%
e^{2\beta\lambda T}+K_{1}\left(  T\right)  .\nonumber
\end{align}
Therefore, the solution exists on any given interval $[0,T]$ and
\begin{equation}
\{u_{n}\}\text{ is bounded in }L^{\infty}(0,T;L^{2}(\Omega
)).\label{firstbounded}%
\end{equation}

Now, we multiply the equation (\ref{problem1}) by $\dfrac{du_{n}}{dt}$ to
obtain
\[
\Vert\frac{du_{n}}{dt}(t)\Vert_{L^{2}}^{2}+a(\Vert u_{n}\Vert_{H_{0}^{1}}%
^{2})\frac{1}{2}\frac{d}{dt}\Vert u_{n}\Vert_{H_{0}^{1}}^{2}=\frac{d}{dt}%
\int_{\Omega}\lambda\mathcal{F}(u_{n})dx+(h(t),\frac{du_{n}}{dt}).
\]
Introducing
\begin{equation}
A(s)=\int_{0}^{s}a(r)dr\label{funcionA}%
\end{equation}
we have
\[
\frac{1}{2}\Vert\frac{du_{n}}{dt}(t)\Vert_{L^{2}}^{2}+\frac{d}{dt}\left(
\frac{1}{2}A(\Vert u_{n}\Vert_{H_{0}^{1}}^{2})-\int_{\Omega}\lambda
\mathcal{F}(u_{n})dx\right)  \leq\frac{1}{2}\Vert h(t)\Vert_{L^{2}}^{2}.
\]
Integrating the previous expression between $0$ and $t$ we get
\begin{equation}%
\begin{split}
&  \frac{1}{2}A(\Vert u_{n}(t)\Vert_{H_{0}^{1}}^{2})+\lambda\int_{\Omega
}\mathcal{F}(u_{n}(0))dx+\frac{1}{2}\int_{0}^{t}\Vert\frac{d}{ds}u_{n}%
(s)\Vert_{L^{2}}^{2}ds\\
&  \leq\frac{1}{2}A(\Vert u_{n}(0)\Vert_{H_{0}^{1}}^{2})+\lambda\int_{\Omega
}\mathcal{F}(u_{n}(t))dx+\frac{1}{2}\int_{0}^{t}\Vert h(s)\Vert_{L^{2}}^{2}ds.
\end{split}
\label{13bound}%
\end{equation}
By (A6), (\ref{Diss2}) and (\ref{9}) it follows that
\begin{equation}%
\begin{split}
&  \frac{m}{2}\Vert u_{n}(t)\Vert_{H_{0}^{1}}^{2}+\lambda\int_{\Omega
}\mathcal{F}(u_{n}(0))dx+\frac{1}{2}\int_{0}^{t}\Vert\frac{d}{ds}u_{n}%
(s)\Vert_{L^{2}}^{2}ds\\
&  \leq\frac{1}{2}A(\Vert u_{n}(0)\Vert_{H_{0}^{1}}^{2})+\lambda
\widetilde{\beta}\Vert u_{n}(t)\Vert_{L^{2}}^{2}+\lambda\widetilde{\gamma
}|\Omega|+K_{2}(T)\\
&  \leq\frac{1}{2}A(\Vert u_{n}(0)\Vert_{H_{0}^{1}}^{2})+\lambda
\widetilde{\beta}e^{2\beta\lambda T}\Vert u_{n}(0)\Vert_{L^{2}}^{2}+K_{3}(T).
\end{split}
\label{bound14}%
\end{equation}
Since $dim(\Omega)=1$, $H_{0}^{1}(\Omega)\subset L^{\infty}(\Omega),$ so
$u_{n}\left(  0\right)  $ is bounded in $L^{\infty}(\Omega)$. Thus, as $f$
maps bounded sets of $\mathbb{R}$ into bounded ones, $\mathcal{F}\left(
u_{n}\left(  0\right)  \right)  $ is bounded in $L^{\infty}(\Omega)$ as well.
Therefore, we deduce that
\[
\{u_{n}\}\text{ is bounded in }L^{\infty}(0,T;H_{0}^{1}(\Omega))
\]
and
\begin{equation}
\frac{du_{n}}{dt}\text{ is bounded in }L^{2}(0,T;L^{2}(\Omega
)).\label{dudtbounded}%
\end{equation}
Using again the embedding $H_{0}^{1}(\Omega)\subset L^{\infty}(\Omega)$ we
obtain that $u_{n}$ is bounded in the space $L^{\infty}(0,T;L^{\infty}%
(\Omega))$. Thus,
\begin{equation}
f(u_{n})\text{ is bounded in }L^{\infty}(0,T;L^{\infty}(\Omega
)).\label{boundedf}%
\end{equation}
Also, we deduce that $\Vert u_{n}(t)\Vert_{H_{0}^{1}}^{2}$ is uniformly
bounded in $[0,T]$ and then by the continuity of the function $a\left(
\text{\textperiodcentered}\right)  $ we get that the sequence $a\left(
\left\Vert u_{n}\left(  t\right)  \right\Vert _{H_{0}^{1}}^{2}\right)  $ is
also uniformly bounded in $[0,T]$.

Finally, multiplying (\ref{1.6}) by $\lambda_{j}\gamma_{ni}(t)$ and summing
from $i=1$ to $n$ we obtain
\[
\frac{1}{2}\frac{d}{dt}\Vert u_{n}\Vert_{H_{0}^{1}}^{2}+m\Vert\Delta
u_{n}\Vert_{L^{2}}^{2}\leq\lambda(f(u_{n}),-\Delta u_{n})+(h(t),-\Delta u).
\]
By (\ref{boundedf}) and applying the Young inequality, we get
\[
\frac{1}{2}\frac{d}{dt}\Vert u_{n}\Vert_{H_{0}^{1}}^{2}+m\Vert\Delta
u_{n}\Vert_{L^{2}}^{2}\leq\frac{\lambda^{2}}{m}\Vert f(u_{n})\Vert_{L^{2}}%
^{2}+\frac{m}{4}\Vert\Delta u_{n}\Vert_{L^{2}}^{2}+\frac{1}{m}\Vert
h(t)\Vert_{L^{2}}^{2}+\frac{m}{4}\Vert\Delta u\Vert_{L^{2}}^{2}.
\]
Integrating the previous expression between $0$ and $t$, it follows that
\[
\Vert u_{n}(t)\Vert_{H_{0}^{1}}^{2}+m\int_{0}^{t}\Vert\Delta u_{n}%
(s)\Vert_{L^{2}}^{2}ds\leq\Vert u_{n}(0)\Vert_{H_{0}^{1}}^{2}+\frac
{2\lambda^{2}}{m}\int_{0}^{t}\Vert f(u_{n}(s))\Vert_{L^{2}}^{2}ds+\frac{2}%
{m}\int_{0}^{t}\Vert h(s)\Vert_{L^{2}}^{2}ds.
\]
Taking into account (\ref{boundedf}), the last inequality implies that
\begin{equation}
u_{n}\text{ is bounded in }L^{2}(0,T;D(A)),\label{Aubounded}%
\end{equation}
so $\{-\Delta u_{n}\}$ and $\{a(\Vert u_{n}\Vert_{H_{0}^{1}}^{2})\Delta
u_{n}\}$ are bounded in $L^{2}(0,T;L^{2}(\Omega))$.

As a consequence, there exists $u\in L^{\infty}(0,T;H_{0}^{1}(\Omega))$ and a
subsequence $u_{n}$ (relabeled the same) such that
\begin{equation}%
\begin{split}
u_{n}  &  \overset{\ast}{\rightharpoonup}u\text{ in }L^{\infty}(0,T;H_{0}%
^{1}(\Omega)),\\
u_{n}  &  \rightharpoonup u\text{ in }L^{2}(0,T;D(A)),\\
f(u_{n})  &  \overset{\ast}{\rightharpoonup}\chi\text{ in }L^{\infty
}(0,T;L^{\infty}(\Omega)),\\
a(\Vert u_{n}\Vert_{H_{0}^{1}}^{2})  &  \overset{\ast}{\rightharpoonup}b\text{
in }L^{\infty}(0,T),
\end{split}
\label{convergences}%
\end{equation}
where $\rightharpoonup$ ($\overset{\ast}{\rightharpoonup}$) stands for the
weak (weak star) convergence. By (\ref{dudtbounded}) and (\ref{Aubounded}) the
Aubin-Lions Compactness Lemma gives that $u_{n}\rightarrow u$ in
$L^{2}(0,T;H_{0}^{1}(\Omega))$, so $u_{n}(t)\rightarrow u(t)$ in $H_{0}%
^{1}(\Omega)$ a.e. on $(0,T).$ Consequently, there exists a subsequence
$u_{n}$, relabelled the same, such that $u_{n}(t,x)\rightarrow u(t,x)$ a.e. in
$\Omega\times(0,T).$

Moreover, thanks to the inequality
\[
\left\Vert u_{n}(t_{2})-u_{n}(t_{1})\right\Vert _{L^{2}}^{2}=\left\Vert
\int_{t_{1}}^{t_{2}}\frac{d}{dt}u_{n}(s)ds\right\Vert _{L^{2}}^{2}\leq
\Vert\frac{d}{dt}u_{n}\Vert_{L^{2}(0,T;L^{2}(\Omega))}^{2}\ |t_{2}-t_{1}%
|\quad\forall t_{1},t_{2}\in\lbrack0,T],
\]
(\ref{bound14}), (\ref{dudtbounded}) and $H_{0}^{1}(\Omega)\subset\subset
L^{2}(\Omega)$, the Ascoli-Arzel\`{a} theorem implies that $\{u_{n}\}$
converges strongly in $C([0,T];L^{2}(\Omega))$ for all $T>0$. Therefore, we
obtain from (\ref{bound14}) that $u_{n}(t)\rightharpoonup u(t)\text{ in }%
H_{0}^{1}(\Omega)$, for any $t\geq0$.

Also, by (\ref{convergences}) we have that $P_{n}f(u_{n}))\rightharpoonup\chi$
in $L^{q}(0,T;L^{q}(\Omega))$ for any $q\geq1$ (see \cite[p.224]{robinson}).
Since $f$ is continuous, it follows that $f(u_{n}(t,x))\rightarrow f(u(t,x))$
a.e. in $\Omega\times(0,T).$ Therefore, in view of (\ref{convergences}), by
\cite[Lemma 1.3]{lions} we have that $\chi=f(u)$.

As a consequence, by the continuity of $a$ we get that
\[
a(\Vert u_{n}(t)\Vert_{H_{0}^{1}}^{2})\rightarrow a(\Vert u(t)\Vert_{H_{0}%
^{1}}^{2})\ \ \text{ a.e. on }(0,T).
\]
Since the sequence is uniformly bounded, by Lebesgue's theorem this
convergence takes place in $L^{2}(0,T),$ so $b=a(\Vert u\Vert_{H_{0}^{1}}%
^{2})$. Thus,
\[
a(\Vert u_{n}\Vert_{H_{0}^{1}}^{2})\Delta u_{n}\rightharpoonup a(\Vert
u\Vert_{H_{0}^{1}}^{2})\Delta u,\ \ \ \text{ in }\ \ \ L^{2}(0,T;L^{2}%
(\Omega)).
\]

Therefore, we can pass to the limit to conclude that $u$ is a strong solution.

It remains to show that $u(0)=u_{0}$ which makes sense since $u\in
C([0,T];H_{0}^{1}(\Omega))$ (see Remark 4). Indeed, let be $\phi\in
C^{1}([0,T];H_{0}^{1}(\Omega))$ with $\phi(T)=0,\ \phi(0)\not =0$. We multiply
the equation in (\ref{problem1}) and (\ref{1.6}) by $\phi$ and integrate by
parts in the $t$ variable to obtain that%
\begin{align}
& \int_{0}^{T}\left(  -(u(t),\phi^{\prime}(t))-a(\Vert u(t)\Vert_{H_{0}^{1}%
}^{2})(\Delta u(t),\phi(t))\right)  dt\label{In1}\\
& =\int_{0}^{T}(\lambda f(u(t))+h(t),\phi(t))dt+(u(0),\phi(0)),\nonumber
\end{align}%
\begin{align}
& \int_{0}^{T}\left(  -(u_{n}(t),\phi^{\prime}(t))-a(\Vert u_{n}%
(t)\Vert_{H_{0}^{1}}^{2})(\Delta u_{n}(t),\phi(t))\right)  dt\label{In2}\\
& =\int_{0}^{T}(\lambda f(u_{n}(t))+h(t),\phi(t))dt+(u_{n}(0),\phi
(0)).\nonumber
\end{align}
In view of the previous convergences, we can pass to the limit in (\ref{In2}).
Taking into account (\ref{In1}) and bearing in mind $u_{n}(0)=P_{n}%
u_{0}\rightarrow u_{0}$, since $\phi\left(  0\right)  \in H_{0}^{1}(\Omega)$
is arbitrary, we infer that $u(0)=u_{0}$.
\end{proof}

\section{Existence and structure of attractors}

In this section, we will prove the existence of a global attractor for the
semiflow generated by strong solutions in the autonomous case. Thus, the
function $h$ will be an independent of time function satisfying (A10) instead
of (A9). Also, we will establish that the attractor is equal to the unstable
set of the stationary points (see the definition in (\ref{UnstableSet})).

Throughout this section, for a metric space $X$ with metric $d$ we will denote
by $dist_{X}\left(  C,D\right)  $ the \ Hausdorff semidistance from $C$ to
$D$, that is,
\[
dist_{X}(C,D)=\sup_{c\in C}\inf_{d\in D}\rho\left(  c,d\right)  .
\]

Let us consider the phase space $X=H_{0}^{1}\left(  \Omega\right)  $ and the
sets
\[
K\left(  u_{0}\right)  =\{u(\cdot):u\text{ is a strong solution of
(\ref{problem1}) such that }u\left(  0\right)  =u_{0}\},
\]%
\[
\mathcal{R=\cup}_{u_{0}\in X}K\left(  u_{0}\right)  .
\]
Denote by $P(X)$ the class of nonempty subsets of $X$. We define the (possibly
multivalued) map $G:\mathbb{R}^{+}\times X\rightarrow P(X)$ by
\begin{equation}
G(t,u_{0})=\{u(t):u\in\mathcal{R}\text{ and }u(0)=u_{0}\}. \label{defG}%
\end{equation}

In order to study the map $G$ let us consider the following axiomatic
properties of the set $\mathcal{R}$:

\begin{itemize}
\item[(K1)] For every $x\in X$ there is $\phi\in\mathcal{R}$ satisfying
$\phi(0)=x$.

\item[(K2)] $\phi_{\tau}(\cdot):=\phi(\cdot+\tau)\in\mathcal{R}$ for every
$\tau\geq0$ and $\phi\in\mathcal{R}$ (translation property).

\item[(K3)] Let $\phi_{1},\phi_{2}\in\mathcal{R}$ be such that $\phi
_{2}(0)=\phi_{1}(s)$ for some $s>0$. Then, the function $\phi$ defined by
\[
\phi(t)=\left\{
\begin{array}
[c]{l}%
\phi_{1}(t)\quad0\leq t\leq s,\\
\phi_{2}(t-s)\quad s\leq t,
\end{array}
\right.
\]
belongs to $\mathcal{R}$ (concatenation property).

\item[(K4)] For every sequence $\{\phi^{n}\}\subset\mathcal{R}$ satisfying
$\phi^{n}(0)\rightarrow x_{0}$ in $X$, there is a subsequence $\{\phi^{n_{k}%
}\}$ and $\phi\in\mathcal{R}$ such that $\phi^{n_{k}}(t)\rightarrow\phi(t)$
for every $t\geq0$.
\end{itemize}

Assuming conditions (A1), (A6), (A10) and (\ref{Diss}) property (K1) follows
from Theorem \ref{existence}, whereas (K2)-(K3) can be proved easily using
equality (\ref{EquationRegular}). By \cite[Proposition 2]{caraballorubio} or
\cite[Lemma 9]{kapustyanpankov} we know that $\mathcal{R}$ fulfilling (K1) and
(K2) gives rise to a multivalued semiflow $G$ through (\ref{defG}) (m-semiflow
for short), which means that:

\begin{itemize}
\item $G(0,x)=x$ for all $x\in X;$

\item $G(t+s,x)\subset G(t,G(s,x))$ for all $t,s\geq0$ and $x\in X$.
\end{itemize}

Moreover, (K3) implies that the m-semiflow is strict, that is,
$G(t+s,x)=G(t,G(s,x))$ for all $t,s\geq0$ and $x\in X$.

We will show first that the m-semiflow $G$ possesses a bounded absorbing set
in the space $L^{2}\left(  \Omega\right)  $ and that property (K4) is satisfied.

\begin{lemma}
\label{lemma4} Assume conditions (A1), (A6), (A10) and (\ref{Diss}). Given
$\{u^{n}\}\subset\mathcal{R},$ $u^{n}(0)\rightarrow u_{0}$ weakly in
$H_{0}^{1}(\Omega)$, there exists a subsequence of $\{u^{n}\}$ (relabeled the
same) and $u\in K(u_{0})$ such that
\[
u^{n}(t)\rightarrow u(t)\mathit{\ }\text{\textit{ in }}H_{0}^{1}%
(\Omega),\ \forall t>0.
\]
Also, if $u^{n}(0)\rightarrow u_{0}$ strongly in $H_{0}^{1}(\Omega)$, then for
$t_{n}\rightarrow0$ we get $u^{n}(t_{n})\rightarrow u_{0}$ strongly in
$H_{0}^{1}(\Omega)$.
\end{lemma}

\begin{proof}
Since $\dfrac{du^{n}}{dt}\in L^{2}(0,T;L^{2}(\Omega))$ and $u^{n}\in
L^{2}(0,T;H_{0}^{1}(\Omega))$, we have by \cite[pg. 102]{sellyou} that
\begin{equation}
\dfrac{d}{dt}\Vert u^{n}\Vert_{H_{0}^{1}}^{2}=2(-\Delta u^{n},u_{t}%
^{n})\ \text{ for a.a. }t \label{equality}%
\end{equation}
and $u^{n}\in C([0,T];H_{0}^{1}(\Omega))$. Also, as $f(u^{n})\in L^{\infty
}(0,T;L^{\infty}(\Omega))$, by regularization one can show that $(F(u^{n}%
(t)),1)$ is an absolutely continuous function on $[0,T]$ and
\begin{equation}
\dfrac{d}{dt}(F(u^{n}(t)),1)=(f(u^{n}(t)),\dfrac{du^{n}}{dt})\ \text{ for a.a.
}t>0. \label{FDeriv}%
\end{equation}

By a similar argument as in Theorem \ref{existence}, there is a subsequence of
$u^{n}$ such that
\begin{equation}%
\begin{split}
u^{n} &  \text{ is bounded in }L^{\infty}(0,T;L^{\infty}(\Omega)),\\
u^{n} &  \text{ is bounded in }L^{\infty}(0,T;H_{0}^{1}(\Omega)),\\
f(u^{n}) &  \text{ is bounded in }L^{\infty}(0,T;L^{\infty}(\Omega)),\\
u^{n} &  \text{ is bounded in }L^{2}(0,T;D(A)).
\end{split}
\label{35}%
\end{equation}

Therefore, arguing as in the proof of Theorem \ref{existence}, there exists
$u\in K\left(  u_{0}\right)  $ and a subsequence $u^{n}$, relabelled the same,
such that
\begin{equation}%
\begin{split}
u^{n}  &  \overset{\ast}{\rightharpoonup}u\text{ in }L^{\infty}(0,T;H_{0}%
^{1}(\Omega))\\
u_{n}  &  \rightharpoonup u\text{ in }L^{2}(0,T;D(A))\\
f(u^{n})  &  \overset{\ast}{\rightharpoonup}f(u)\text{ in }L^{\infty
}(0,T;L^{\infty}(\Omega))\\
\dfrac{du^{n}}{dt}  &  \rightharpoonup\dfrac{du}{dt}\text{ in }L^{2}%
(0,T;L^{2}(\Omega))\\
a(\Vert u^{n}\Vert_{H_{0}^{1}}^{2})\Delta u^{n}  &  \rightharpoonup a(\Vert
u\Vert_{H_{0}^{1}}^{2})\Delta u\text{ in }L^{2}(0,T;L^{2}(\Omega)),\\
u^{n}  &  \rightarrow u\text{ in }L^{2}(0,T;H_{0}^{1}(\Omega)),\\
u^{n}  &  \rightarrow u\text{ in }C([0,T],L^{2}(\Omega)),\\
u^{n}(t)  &  \rightharpoonup u(t)\text{ in }H_{0}^{1}(\Omega)\quad\forall
t\in(0,T].
\end{split}
\label{convergences1}%
\end{equation}

We also need to prove that $u^{n}(t)\rightarrow u(t)$ in $H_{0}^{1}(\Omega)$
for all $t\in(0,T]$. For this end, we multiply (\ref{problem1}) by $u_{t}^{n}$
and using (A10), (\ref{equality}) and (\ref{35}) we have
\[
\frac{1}{2}\left\Vert \frac{du^{n}}{dt}\right\Vert _{L^{2}}^{2}+\frac{d}%
{dt}\left(  \frac{1}{2}A(\Vert u^{n}(t)\Vert_{H_{0}^{1}}^{2})\right)  \leq C.
\]
Thus, we obtain
\[
A(\Vert u^{n}(t)\Vert_{H_{0}^{1}}^{2})\leq A(\Vert u^{n}(s)\Vert_{H_{0}^{1}%
}^{2})+2C(t-s),\quad t\geq s\geq0.
\]
Since this inequality is also true for $u(\cdot)$, the functions
$Q_{n}(t)=A(\Vert u^{n}(t)\Vert_{H_{0}^{1}}^{2})-2Ct,$ $Q(t)=A(\Vert
u(t)\Vert_{H_{0}^{1}}^{2})-2Ct$ are continuous and non-increasing in $[0,T]$.
Moreover, from (\ref{convergences1}) we deduce that
\[
Q_{n}(t)\rightarrow Q(t)\quad\text{ for a.e. }t\in(0,T).
\]
Take $0<t\leq T$ $\ $and $0<t_{j}<t$ such that $t_{j}\rightarrow t$ and
$Q_{n}(t_{j})\rightarrow Q(t_{j})$ for all $j$. Then
\[
Q_{n}(t)-Q(t)\leq Q_{n}(t_{j})-Q(t)\leq|Q_{n}(t_{j})-Q(t_{j})|+|Q(t_{j}%
)-Q(t)|.
\]
For any $\delta>0$ there exist $j(\delta)$ and $N(j(\delta))$ such that
$Q_{n}(t)-Q(t)\leq\delta$ if $n\geq N$. Then $\limsup Q_{n}(t)\leq Q(t),$ so
$\limsup\Vert u^{n}(t)\Vert_{H_{0}^{1}}^{2}\leq\Vert u(t)\Vert_{H_{0}^{1}}%
^{2}$, which follows by contradiction using the continuity of the function
$A(s)$. As $u^{n}(t)\rightarrow u(t)$ weakly in $H_{0}^{1}(\Omega)$ implies
that $\liminf\Vert u^{n}(t)\Vert_{H_{0}^{1}}^{2}\geq\Vert u(t)\Vert_{H_{0}%
^{1}}^{2}$, we obtain
\[
\Vert u^{n}(t)\Vert_{H_{0}^{1}}^{2}\rightarrow\Vert u(t)\Vert_{H_{0}^{1}}^{2},
\]
so that $u^{n}(t)\rightarrow u(t)$ strongly in $H_{0}^{1}(\Omega)$.

Finally, if $u^{n}(0)\rightarrow u_{0}$ strongly in $H_{0}^{1}(\Omega)$ and we
take $t_{n}\rightarrow0$, then%
\[
Q_{n}(t_{n})-Q(0)\leq Q_{n}(0)-Q(0)=A(\Vert u^{n}\left(  0\right)
\Vert_{H_{0}^{1}}^{2})-A(\Vert u_{0}\Vert_{H_{0}^{1}}^{2})\rightarrow0,
\]
so $\limsup Q_{n}(t_{n})\leq Q(0)$. Repeating the above argument, we infer
that $u^{n}(t_{n})\rightarrow u_{0}$ strongly in $H_{0}^{1}(\Omega)$.
\end{proof}

\begin{corollary}
\label{PropK4} Assume the conditions of Lemma \ref{lemma4}. Then the set
$\mathcal{R}$ satisfies condition $(K4)$.
\end{corollary}

The map $t\mapsto G(t,x)$ is said to be upper semicontinuous if for every
$x\in X$ and for an arbitrary neighborhood $O(G(t,x))$ in $X$ there is
$\delta>0$ such that as soon as $d(y,x)<\delta$, we have $G(t,y)\subset O$.

\begin{proposition}
\label{uppersemicontinuity} Assume the conditions of Lemma \ref{lemma4}. The
multivalued semiflow $G$ is upper semicontinuous for all $t\geq0$. Also, it
has compact values.
\end{proposition}

\begin{proof}
By contradiction let us assume that there exist $t\geq0,\ u_{0}\in H_{0}%
^{1}(\Omega)$, a neighbourhood $O(G(t,u_{0}))$ and sequences $\{y_{n}%
\},\ \{u_{0}^{n}\}$ such that $y_{n}\in G(t,u_{0}^{n})$, $u_{0}^{n}$ converges
strongly to $u_{0}$ in $H_{0}^{1}(\Omega)$ and $y_{n}\notin O(G(t,u_{n}))$ for
all $n\in\mathbb{N}$. Thus, there exist $u^{n}\in K(u_{0}^{n})$ such that
$y_{n}=u^{n}(t)$. From Lemma \ref{lemma4} there exists a subsequence of
$y_{n}$ which converges to some $y\in G(t,u_{0})$. This contradicts
$y_{n}\notin O(G(t,u_{0}))$ for any $n\in\mathbb{N}$.
\end{proof}

\bigskip

In order to prove the existence of an absorbing set in the space $L^{2}\left(
\Omega\right)  $ we need to use the stronger condition (A5) instead of
(\ref{Diss}).

\begin{proposition}
\label{abosorbinginL2firstcase}Assume that conditions (A1), (A5), (A6) and
(A10) hold. Then the m-semiflow $G$ has a bounded absorbing set in
$L^{2}\left(  \Omega\right)  $, that is, there exists a constant $K>0$ such
that for any $R>0$ there is a time $t_{0}=t_{0}(R)$ such that
\begin{equation}
\Vert y\Vert_{L^{2}}\leq K\quad\text{ for all }\quad t\geq t_{0},\text{ }y\in
G(t,u_{0}), \label{9.1}%
\end{equation}
where $\left\Vert u_{0}\right\Vert _{L^{2}}\leq R.$ Moreover, there is $L>0$
such that%
\begin{equation}
\int_{t}^{t+1}\Vert u(s)\Vert_{H_{0}^{1}}^{2}ds\leq L\quad\text{ for all}\quad
t\geq t_{0}\text{, }u\in K\left(  u_{0}\right)  . \label{31}%
\end{equation}

\end{proposition}

\begin{proof}
Multiplying equation (\ref{problem1}) by $u$ and using (A6) and (\ref{condf})
we get
\begin{align}
\frac{1}{2}\frac{d}{dt}\Vert u(t)\Vert_{L^{2}}^{2}+m\Vert u(t)\Vert_{H_{0}%
^{1}}^{2}  & \leq(f(u),u)+(h,u)\label{9.3}\\
& \leq m_{\varepsilon}|\Omega|+\varepsilon\Vert u(t)\Vert_{L^{2}}^{2}+\frac
{1}{2\lambda_{1}m}\Vert h\Vert_{L^{2}}^{2}+\frac{\lambda_{1}m}{2}\Vert
u\Vert_{L^{2}}^{2}.\nonumber
\end{align}
Using the Poincar\'{e} inequality it follows that
\[
\frac{d}{dt}\Vert u\Vert_{L^{2}}^{2}\leq2m_{\varepsilon}|\Omega|+2(\varepsilon
-\frac{m}{2}\lambda_{1})\Vert u(t)\Vert_{L^{2}}^{2}+\frac{1}{\lambda_{1}%
m}\Vert h\Vert_{L^{2}}^{2}=-\delta\Vert u(t)\Vert_{L^{2}}^{2}+\kappa,
\]
where $\delta=m\lambda_{1}-2\varepsilon$, $\kappa=2m_{\varepsilon}%
|\Omega|+\frac{1}{\lambda_{1}m}\Vert h\Vert_{L^{2}}^{2}$. We take
$\varepsilon>0$ small enough, so $\delta>0$. Then Gronwall's lemma gives
\begin{equation}
\Vert u(t)\Vert_{L^{2}}^{2}\leq\Vert u(0)\Vert_{L^{2}}^{2}e^{-\delta t}%
+\frac{\kappa}{\delta}.\label{Inequ}%
\end{equation}
Hence, taking
\[
t\geq t_{0}=\frac{1}{\delta}\ln\left(  \frac{\delta R^{2}}{\kappa}\right)
\]
we get (\ref{9.1}) for $K=\sqrt{\frac{2\kappa}{\delta}}.$

On the other hand, using again the Poincar\'{e} inequality from (\ref{9.3}) we
get
\[
\frac{d}{dt}\Vert u(t)\Vert_{L^{2}}^{2}+\left(  \frac{m\lambda_{1}%
-2\varepsilon}{\lambda_{1}}\right)  \Vert u(t)\Vert_{H_{0}^{1}}^{2}\leq\kappa
\]
and integrating from $t$ to $t+1$ we obtain%
\[
\left(  \frac{m\lambda_{1}-2\varepsilon}{\lambda_{1}}\right)  \int_{t}%
^{t+1}\Vert u(s)\Vert_{H_{0}^{1}}^{2}ds\leq\Vert u(t)\Vert_{L^{2}}^{2}%
+\kappa.
\]
Therefore, applying (\ref{9.1}), (\ref{31}) follows.
\end{proof}

\bigskip

Further, in order to obtain an absorbing set in $H_{0}^{1}\left(
\Omega\right)  $ we need to assume additionally that either the function
$a\left(  \text{\textperiodcentered}\right)  $ is bounded above or that it is non-decreasing.

\begin{proposition}
\label{absorbingset} Assume the conditions in Proposition
\ref{abosorbinginL2firstcase} and that either (A7) or (A8) holds true. Then
there exists an absorbing set $B_{1}$ for $G$, which is compact in $H_{0}%
^{1}(\Omega)$.
\end{proposition}

\begin{proof}
In view of Proposition \ref{abosorbinginL2firstcase} we have an absorbing set
$B_{0}$ in $L^{2}(\Omega)$. Let $K>0$ be such that $\Vert y\Vert\leq K$ for
all $y\in B_{0}$.

Multiplying (\ref{problem1}) by $u$ and using (\ref{condf}) and (\ref{Inequ})
we get
\begin{align*}
\frac{d}{dt}\Vert u(t)\Vert_{L^{2}}^{2}+a\left(  \Vert u(t)\Vert_{H_{0}^{1}%
}^{2}\right)  \Vert u(t)\Vert_{H_{0}^{1}}^{2} &  \leq2m_{\varepsilon}%
|\Omega|+2\varepsilon\Vert u(t)\Vert_{L^{2}}^{2}+\frac{1}{\lambda_{1}m}\Vert
h\Vert_{L^{2}}^{2}\\
&  \leq K_{1}+K_{2}\Vert u(0)\Vert_{L^{2}}^{2}.
\end{align*}
Thus, integrating between $t$ and $t+r$, $0<r\leq1$, we deduce by using
(\ref{Inequ}) again that%
\begin{equation}%
\begin{split}
&  \Vert u(t+r)\Vert_{L^{2}}^{2}+\int_{t}^{t+r}a\left(  \Vert u(s)\Vert
_{H_{0}^{1}}^{2}\right)  \Vert u(s)\Vert_{H_{0}^{1}}^{2}ds\\
&  \leq K_{1}+K_{2}\Vert u(0)\Vert_{L^{2}}^{2}+\Vert u(t)\Vert_{L^{2}}^{2}\leq
K_{3}\Vert u(0)\Vert_{L^{2}}^{2}+K_{4}.
\end{split}
\label{EstH1}%
\end{equation}
Also, if $p>2$ in (A5), we multiply again by (\ref{problem1}) by $u$ and use
(\ref{A52}) and (A6) to obtain%
\[
\frac{1}{2}\frac{d}{dt}\Vert u(t)\Vert_{L^{2}}^{2}+\frac{m}{2}\Vert
u(t)\Vert_{H_{0}^{1}}^{2}+C_{4}\left\Vert u\left(  t\right)  \right\Vert
_{L^{p}}^{p}\leq C_{3}+\frac{1}{2\lambda_{1}m}\Vert h\Vert_{L^{2}}^{2}.
\]
Integrating over $\left(  t,t+r\right)  $ we have%
\begin{equation}
\Vert u(t+r)\Vert_{L^{2}}^{2}+2C_{4}\int_{t}^{t+r}\left\Vert u\left(
s\right)  \right\Vert _{L^{p}}^{p}ds\leq K_{5}+\Vert u(t)\Vert_{L^{2}}^{2}\leq
K_{6}+\Vert u(0)\Vert_{L^{2}}^{2}.\label{EstLp}%
\end{equation}
If we assume (A7), by (\ref{EstH1}) and (A6) we have that
\begin{equation}
\int_{t}^{t+r}A(\Vert u(s)\Vert_{H_{0}^{1}}^{2})ds\leq\int_{t}^{t+r}M_{1}\Vert
u(s)\Vert_{H_{0}^{1}}^{2}ds\leq K_{7}(1+\Vert u(0)\Vert_{L^{2}}^{2}%
).\label{EstA1}%
\end{equation}
If we assume (A8), by (\ref{EstH1}) we obtain%
\begin{align}
\int_{t}^{t+r}A(\Vert u(s)\Vert_{H_{0}^{1}}^{2})ds &  =\int_{t}^{t+r}\int%
_{0}^{\Vert u(s)\Vert_{H_{0}^{1}}^{2}}a\left(  r\right)  drds\nonumber\\
&  \leq\int_{t}^{t+r}a\left(  \Vert u(s)\Vert_{H_{0}^{1}}^{2}\right)  \Vert
u(s)\Vert_{H_{0}^{1}}^{2}ds\leq K_{3}\Vert u(0)\Vert_{L^{2}}^{2}%
+K_{4}.\label{EstA2}%
\end{align}
On the other hand, by (\ref{condFA}) we get
\begin{equation}
-\int_{\Omega}F\left(  u\left(  t\right)  \right)  dx\geq-\widetilde{C}%
\int_{\Omega}(1+|u\left(  t\right)  |^{p})dx.\label{EstF}%
\end{equation}
Using (\ref{equality}) and (\ref{FDeriv}) we can argue as in Theorem
\ref{existence} to obtain
\[
\frac{1}{2}\Vert u_{t}\Vert_{L^{2}}^{2}+\frac{d}{dt}\left(  \frac{1}{2}A(\Vert
u(t)\Vert_{H_{0}^{1}}^{2}-\int_{\Omega}\lambda\mathcal{F}(u_{n})dx\right)
\leq\frac{1}{2}\Vert h\Vert_{L^{2}}^{2}.
\]
Since (\ref{EstLp})-(\ref{EstF}) imply that%
\[
\int_{t}^{t+r}\left(  \frac{1}{2}A(\Vert u(s)\Vert_{H_{0}^{1}}^{2}%
-\int_{\Omega}\lambda\mathcal{F}(u\left(  s\right)  )dx\right)  ds\leq
K_{8}+K_{9}\Vert u(0)\Vert_{L^{2}}^{2},
\]
we can apply the Uniform Gronwall Lemma to get
\[
\frac{1}{2}A(\Vert u(t+r)\Vert_{H_{0}^{1}}^{2})-\int_{\Omega}\lambda
\mathcal{F}(u\left(  t+r\right)  )dx\leq\frac{K_{8}+K_{9}\Vert u(0)\Vert
_{L^{2}}^{2}}{r}+K_{10},\quad\text{ for all }t\geq0,
\]
so by condition $(A6)$, (\ref{condF}) and (\ref{Inequ}) it follows that
\[
\Vert u(t+1)\Vert_{H_{0}^{1}}^{2}\leq K_{11}+K_{12}\Vert u(0)\Vert_{L^{2}}%
^{2},
\]
for all $t\geq0$. In particular,
\[
\Vert u(1)\Vert_{H_{0}^{1}}^{2}\leq K_{11}+K_{12}\Vert u(0)\Vert_{L^{2}}^{2},
\]
for any strong solution $u(\cdot)$ with initial condition $u(0)$.

For any $u_{0}\in H_{0}^{1}\left(  \Omega\right)  $ with $\left\Vert
u_{0}\right\Vert _{H_{0}^{1}}\leq R$ and any $u\in\mathcal{R}$ such that
$u\left(  0\right)  =u_{0}$, the semiflow property $G(t+1,u_{0})\subset
G(1,G(t,u_{0}))$ and $G(t,u_{0})\subset B_{0}$, if $t\geq t_{0}\left(
R\right)  ,$ imply that
\[
\Vert u\left(  t+1\right)  \Vert_{H_{0}^{1}}^{2}\leq C(1+K^{2})\text{ }\forall
t\geq t_{0}\left(  R\right)  .
\]
Then there exists $M>0$ such that the closed ball $B_{M}$ in $H_{0}^{1}\left(
\Omega\right)  $ centered at $0$ with radius $M$ is absorbing for $G$.

By Lemma \ref{lemma4} the set $B_{1}=\overline{G(1,B_{M})}$ is an absorbing
set which is compact in $H_{0}^{1}\left(  \Omega\right)  $.
\end{proof}

\bigskip

Given an m-semiflow $G,$ a set $B\subset X$ is said to be negatively
(positively) invariant if $B\subset G(t,B)$ ($G(t,B)\subset B$) for all
$t\geq0$, and strictly invariant (or, simply, invariant) if it is both
negatively and positively invariant.

We recall that a set $\mathcal{A}\subset X$ is called a global attractor for
the m-semiflow $G$ if it is negatively invariant and attracts all bounded
subsets, i.e., $dist_{X}(G(t,B),\mathcal{A})\rightarrow0$ as $t\rightarrow
+\infty$. When $\mathcal{A}$ is compact, it is the minimal closed attracting
set \cite[Remark 5]{melnikvalero}.

\begin{theorem}
\label{existenceatracttorcase2}Assume the conditions of Proposition
\ref{absorbingset}. Then the multivalued semiflow $G$ possesses a global
compact invariant attractor $\mathcal{A}$.
\end{theorem}

\begin{proof}
From Propositions \ref{uppersemicontinuity} and \ref{absorbingset} we deduce
that the multivalued semiflow $G$ is upper semicontinuous with closed values
and the existence of an absorbing which is compact in $H_{0}^{1}\left(
\Omega\right)  $. Therefore, by \cite[Theorem 4 and Remark 8]{melnikvalero}
the existence of the global invariant attractor and its compactness in
$H_{0}^{1}\left(  \Omega\right)  $ follow.
\end{proof}

\vspace{0.4cm}

We recall some concepts which are necessary to study the structure of the
global attractor.

\begin{definition}
\label{definition8} A map $\phi:\mathbb{R}\rightarrow X$ is a complete
trajectory of $\mathcal{R}$ if $\phi(\cdot+s)\mid_{\lbrack0,\infty)}%
\in\mathcal{R}$ for all $s\in\mathbb{R}$. It is a complete trajectory of $G$
if $\phi(t+s)\in G(t,\phi(s))$ for every $s\in\mathbb{R}$, $t\geq0.$

An element $z\in X$ is a fixed point of $\mathcal{R}$ if $\varphi(\cdot)\equiv
z\in\mathcal{R}$. We denote the set of all fixed points by $\mathfrak{R}%
_{\mathcal{R}}$.

An element $z\in X$ is a fixed point of $G$ if $z\in G(t,z)$ for every
$t\geq0$.
\end{definition}

Several properties concerning fixed points, complete trajectories and global
attractors are summarized in the following results \cite{kapustyankasyanov}.

\begin{lemma}
Let (K1)-(K2) hold. Then each fixed point (complete trajectory) of
$\mathcal{R}$ is also a fixed point (complete trajectory) of $G$.

Let (K1)-(K4) hold. Then the fixed points of $\mathcal{R}$ and $G$ are the
same. In addition, a map $\phi:\mathbb{R}\rightarrow X$ is a complete
trajectory of $\mathcal{R}$ if and only if it is continuous and a complete
trajectory of $G$.
\end{lemma}

The standard well-known result in the single-valued case for describing the
attractor as the union of bounded complete trajectories reads in the
multivalued case as follows.

\begin{theorem}
\label{structureattractor} Suppose that (K1)-(K2) are satisfied and that
either (K3) or (K4) holds true. The semiflow $G$ is assumed to have a compact
global attractor $\mathcal{A}$. Then
\begin{equation}
\mathcal{A}=\{\gamma(0):\gamma\in\mathbb{K}\}=\cup_{t\in\mathbb{R}}%
\{\gamma(t):\gamma\in\mathbb{K}\},\label{attractor}%
\end{equation}
where $\mathbb{K}$ stands for the set of all bounded complete trajectories in
$\mathcal{R}$.
\end{theorem}

In view of Theorem \ref{structureattractor}, as $\mathcal{R}$ satisfies (K3)
and (K4) (by Corollary \ref{PropK4}), the global attractor is characterized in
terms of bounded complete trajectories, so (\ref{attractor}) follows.

The set $B$ is said to be weakly invariant if for any $x\in B$ there exists a
complete trajectory $\gamma$ of $\mathcal{R}$ contained in $B$ such that
$\gamma(0)=x$. Characterization (\ref{attractor}) implies that the attractor
$\mathcal{A}$ is weakly invariant.

The set of fixed points $\mathfrak{R}_{\mathcal{R}}$ is characterized as follows.

\begin{lemma}
\label{lemma22} Assume the conditions of Lemma \ref{lemma4}. Let
$\mathfrak{R}$ be the set of $z\in H^{2}(\Omega)\cap H_{0}^{1}(\Omega)$ such
that
\begin{equation}
-a(\Vert z\Vert_{H_{0}^{1}}^{2})\frac{d^{2}z}{dx^{2}}=\lambda f(z)+h\quad
\text{ in }L^{2}(\Omega). \label{eqfixpoint}%
\end{equation}
Then $\mathfrak{R}_{\mathcal{R}}=\mathfrak{R}$.
\end{lemma}

\begin{proof}
If $z\in\mathfrak{R}_{\mathcal{R}},$ then $u(t)\equiv z\in\mathcal{R}.$ Thus,
$u(\cdot)$ satisfies (\ref{EquationRegular}) and $\frac{du}{dt}=0$ in
$L^{2}(0,T;L^{2}(\Omega))$, so (\ref{eqfixpoint}) is satisfied. Let
$z\in\mathfrak{R}$. Then the map $u(t)\equiv z$ satisfies (\ref{eqfixpoint})
for any $t\geq0$ and $\frac{du}{dt}=0$ in $L^{2}(0,T;L^{2}(\Omega))$, so
(\ref{EquationRegular}) holds true.
\end{proof}

\vspace{0.4cm}

Finally, we shall obtain the characterization of the global attractor in terms
of the unstable and stable sets of the stationary points.

\begin{theorem}
\label{theorem12} Assume the conditions of Proposition \ref{absorbingset}.
Then it holds that
\[
\mathcal{A}=M^{+}(\mathfrak{R})=M^{-}(\mathfrak{R}),
\]
where
\begin{equation}
M^{+}(\mathfrak{R})=\{z:\exists\gamma(\cdot)\in\mathbb{K},\ \gamma
(0)=z,\ \text{ dist }_{H_{0}^{1}}(\gamma(t),\mathfrak{R})\rightarrow
0,\ t\rightarrow+\infty\}, \label{StableSet}%
\end{equation}%
\begin{equation}
M^{-}(\mathfrak{R})=\{z:\exists\gamma(\cdot)\in\mathbb{F},\ \gamma
(0)=z,\ \text{ dist }_{H_{0}^{1}}(\gamma(t),\mathfrak{R})\rightarrow
0,\ t\rightarrow-\infty\}, \label{UnstableSet}%
\end{equation}
and $\mathbb{F}$ denotes the set of all complete trajectories of $\mathcal{R}$
(see Definition \ref{definition8}).
\end{theorem}

\begin{remark}
In (\ref{UnstableSet}) it is equivalent to use $\mathbb{K}$ instead of
$\mathbb{F}$ because all the solutions are bounded forward in time.
\end{remark}

\begin{proof}
We consider the function $E:\mathcal{A}\rightarrow\mathbb{R}$
\begin{equation}
E(y)=\frac{1}{2}A(\Vert y\Vert_{H_{0}^{1}}^{2})-\lambda\int_{\Omega
}F(y(x))dx-\int_{\Omega}h(x)y(x)dx. \label{Lyapunov}%
\end{equation}
Note that $E(y)$ is continuous in $H_{0}^{1}(\Omega)$. Indeed, the maps
$y\mapsto\frac{1}{2}A(\Vert y\Vert_{H_{0}^{1}}^{2})$ and $y\mapsto\int%
_{\Omega}h\left(  x\right)  y\left(  x\right)  dx$ are obviously continuous in
$H_{0}^{1}(\Omega)$. On the other hand, by the embedding $H_{0}^{1}%
(\Omega)\subset L^{\infty}(\Omega)$ and using Lebesgue's theorem, the
continuity of $y\rightarrow\int_{\Omega}F(y(x))dx$ follows.

Using (\ref{equality})-(\ref{FDeriv}) and multiplying the equation
(\ref{problem1}) by $\dfrac{du}{dt}$ for any $u\in\mathcal{R}$ we can obtain
the following energy equality
\[
\int_{s}^{t}\Vert\frac{d}{dr}u(r)\Vert_{L^{2}}^{2}dr+E(u(t))=E(u(s))\quad
\text{ for all }t\geq s\geq0.
\]
Hence, $E(u(t))$ is non-increasing and by $(A6)$, (\ref{condF}) and the
boundedness of $\mathcal{A}$, it is bounded from below. Thus
$E(u(t))\rightarrow l$, as $t\rightarrow+\infty$, for some $l\in\mathbb{R}.$

Let $z\in\mathcal{A}$ and $u\in\mathbb{K}$ be such that $u(0)=z$. By
contradiction, suppose the existence of $\varepsilon>0$ and $u(t_{n})$, where
$t_{n}\rightarrow+\infty$, for which $dist_{H_{0}^{1}}(u(t_{n}),\mathfrak{R}%
)>\varepsilon.$ Since $\mathcal{A}$ is compact in $H_{0}^{1}(\Omega)$, we can
take a converging subsequence (relabeled the same) such that $u(t_{n}%
)\rightarrow y$ in $H_{0}^{1}(\Omega)$, where $t_{n}\rightarrow\infty$. By the
continuity of the function $E$, it follows that $E(y)=l$. We will obtain a
contradiction by proving that $y\in\mathfrak{R}$. Define $v_{n}\left(
\text{\textperiodcentered}\right)  =u(\cdot+t_{n})$. By Lemma \ref{lemma4},
there exist $v\in$ $\mathcal{R}$ and a subsequence satisfying $v(0)=y$ and
$v_{n}(t)\rightarrow v(t)$ in $H_{0}^{1}(\Omega)$ for $t\geq0$. Thus, from
$E(v_{n}(t))\rightarrow E(v(t))$ we infer that $E(v\left(  t\right)  )=l$.
Also, $v(\cdot)$ satisfies the energy equality, so that
\[
l+\int_{0}^{t}\Vert v_{r}\Vert_{L^{2}}^{2}dr=E(v\left(  t\right)  )+\int%
_{0}^{t}\Vert v_{r}\Vert_{L^{2}}^{2}dr=E(v(0))=E(y)=l.
\]
Therefore, $\dfrac{dv}{dt}(s)=0$ for a.a. $s$, and then by Lemma \ref{lemma22}
we have $y\in\mathfrak{R}_{\mathcal{R}}=\mathfrak{R}$. As a consequence,
$\mathcal{A}\subset M^{+}(\mathfrak{R})$. The converse inclusion follows from
(\ref{attractor}).

As before, take arbitrary $z\in\mathcal{A}$ and $u\in\mathbb{K}$ satisfying
$u(0)=z$. Since by the embedding $H_{0}^{1}(\Omega)\subset C([0,1])$ the
energy function is bounded from above in $\mathcal{A}$, $E(u(t))\rightarrow
l$, as $t\rightarrow-\infty$, for some $l\in\mathbb{R}$. Suppose that there
are $\varepsilon>0$ and $u(t_{n})$, where $t_{n}\rightarrow+\infty$, such that
$dist_{H_{0}^{1}}(u(-t_{n}),\mathfrak{R})>\varepsilon$. Up to a subsequence we
have that $u(-t_{n})\rightarrow y$ in $H_{0}^{1}(\Omega)$, $E(y)=l$. Moreover,
for $v_{n}(\cdot)=u(\cdot-t_{n})$ there are $v\in\mathcal{R}$ and a
subsequence such that $v(0)=y$ and $v_{n}(t)\rightarrow v(t)$ in $H_{0}%
^{1}(\Omega)$ for $t\geq0$. Therefore, $E(v_{n}(t))\rightarrow E(v(t))$ gives
$E(v(t))=l$ and then by the above arguments we get a contradiction because
$y\in\mathfrak{R}$. Hence, $\mathcal{A}\subset M^{-}(\mathfrak{R})$ and we
deduce the converse inclusion from (\ref{attractor}).
\end{proof}

\bigskip

Finally, we are able to obtain that the global attractor is compact in the
space $C^{1}\left(  [0,1]\right)  $. This property will be important in order
to study a more precise structure of the global attractor in terms of the
stationary points and their heteroclinic connections.

We define the function $w\left(  t\right)  =u\left(  \alpha^{-1}\left(
t\right)  \right)  $, where $\alpha(t)=\int_{0}^{t}a(\Vert u(s)\Vert
_{H_{0}^{1}}^{2})ds$, which is under the conditions of Proposition
\ref{absorbingset} (see \cite{CabMarVal} for more details) a strong solution
to the problem%
\begin{equation}
\left\{
\begin{array}
[c]{l}%
\dfrac{\partial w}{\partial t}-\dfrac{\partial^{2}w}{\partial x^{2}}%
=\dfrac{f(w)+h}{a(\Vert w\Vert_{H_{0}^{1}}^{2})},\ \text{in }(0,\infty
)\times\Omega,\\
w=0\quad\text{on }(0,\infty)\times\partial\Omega,\\
w(0,x)=u_{0}(x)\quad\text{in }\Omega.
\end{array}
\right.  \label{w}%
\end{equation}

Let $V^{2r}=D(A^{r})$, $r\geq0$. We will prove first that the attractor is
compact in any space $V^{2r}$ with $0\leq r<1.$ For this aim we will need the
concept of mild solution. We consider the auxiliary problem{\normalsize
\begin{equation}
\left\{
\begin{array}
[c]{c}%
\dfrac{dv}{dt}+Av(t)=g\left(  t\right)  ,\ t>0,\\
v\left(  0\right)  =u_{0},
\end{array}
\right.  \label{ProblemAf}%
\end{equation}
where }$g\in L_{loc}^{2}\left(  0,+\infty;L^{2}\left(  \Omega\right)  \right)
$. {\normalsize The function $u\in C([0,+\infty),L^{2}\left(  \Omega\right)
)$ is called a mild solution to problem (\ref{ProblemAf}) if%
\begin{equation}
v\left(  t\right)  =e^{-At}u_{0}+\int_{0}^{t}e^{-A\left(  t-s\right)
}g(s)ds\text{, }\forall t\geq0. \label{VarConstants}%
\end{equation}
}In the same way as in Lemma 2 in \cite{Valero} we obtain that a strong
solution to problem (\ref{w}) is a mild solution to problem (\ref{ProblemAf})
with $g\left(  t\right)  =\left(  f(w\left(  t\right)  )+h\right)  /a(\Vert
w\left(  t\right)  \Vert_{H_{0}^{1}}^{2})$.

\begin{lemma}
\label{CompactVr}Assume the conditions of Proposition \ref{absorbingset}. Then
the global attractor $\mathcal{A}$ is compact in $V^{2r}$ for every $0\leq
r<1.$
\end{lemma}

\begin{proof}
Let $z\in\mathcal{A}$ be arbitrary. Since $\mathcal{A}$ is invariant, there
exist $u_{0}\in\mathcal{A}$ and $u\in\mathcal{R}$ such that $z=u\left(
1\right)  $ and $u\left(  t\right)  \in\mathcal{A}$ for all $t\geq0$. Since
$w\left(  t\right)  =u\left(  \alpha^{-1}\left(  t\right)  \right)  $ is a
mild solution of (\ref{ProblemAf}) with $g\left(  t\right)  =\left(
f(w\left(  t\right)  )+h\right)  /a(\Vert w\left(  t\right)  \Vert_{H_{0}^{1}%
}^{2})$, the variation of constants formula (\ref{VarConstants}) gives%
\[
z=w(\alpha\left(  1\right)  )=e^{-A\alpha\left(  1\right)  }u_{0}+\int%
_{0}^{\alpha\left(  1\right)  }e^{-A\left(  \alpha\left(  1\right)  -s\right)
}g(s)ds.
\]
As $\mathcal{A}$ is bounded in $H_{0}^{1}\left(  \Omega\right)  $ (and then in
$L^{\infty}\left(  \Omega\right)  $), condition (A6) and the continuity of $f$
imply that%
\[
\left\Vert u_{0}\right\Vert _{L^{2}}\leq C,\ \left\Vert g\right\Vert
_{L^{\infty}\left(  0,\alpha\left(  1\right)  ;L^{2}\left(  \Omega\right)
\right)  }\leq C,
\]
where $C>0$ does not depend on $z$. The standard estimate $\left\Vert
e^{-At}\right\Vert _{\mathcal{L}(L^{2}\left(  \Omega\right)  ,D(A^{r}))}\leq
M_{r}t^{-r}e^{-at},\ M_{r},a>0$ \cite[Theorem 37.5]{sellyou}, implies that%
\begin{align*}
\left\Vert A^{r}z\right\Vert _{L^{2}}  &  \leq\left\Vert A^{r}e^{-A\alpha
\left(  1\right)  }u_{0}\right\Vert _{L^{2}}+\int_{0}^{\alpha\left(  1\right)
}\left\Vert A^{r}e^{-A\left(  \alpha\left(  1\right)  -s\right)
}g(s)\right\Vert _{L^{2}}ds\\
&  \leq M_{r}e^{-a\alpha\left(  1\right)  }\alpha\left(  1\right)
^{-r}C+M_{r}C\int_{0}^{\alpha\left(  1\right)  }\left(  \alpha\left(
1\right)  -s\right)  ^{-r}ds,
\end{align*}
so $\mathcal{A}$ is bounded in $V^{2r}$ for every $0\leq r<1.$

From the compact embedding $V^{\alpha}\subset V^{\beta},$ for $\alpha>\beta$,
and the fact that $\mathcal{A}$ is closed in any $V^{2r}$ we obtain the result.
\end{proof}

\begin{corollary}
\label{CompactC1}Assume the conditions of Proposition \ref{absorbingset}. Then
the global attractor $\mathcal{A}$ is compact in $C^{1}([0,1]).$
\end{corollary}

\begin{proof}
We obtain by Lemma 37.8 in \cite{sellyou} the continuous embedding%
\[
V^{2r}\subset C^{1}([0,1])\text{ if }r>\frac{3}{4}.
\]
Hence, the statement follows from Lemma \ref{CompactVr}.
\end{proof}

\section{ Fixed points}

In this section we are interested in studying the fixed points of problem
(\ref{problem1}) when $h\equiv0$, that is, the solutions of the boundary-value
problem%
\begin{equation}
\left\{
\begin{array}
[c]{c}%
-a(\Vert u\Vert_{H_{0}^{1}}^{2})\dfrac{d^{2}u}{dx^{2}}=\lambda f(u),\ 0<x<1,\\
u\left(  0\right)  =u\left(  1\right)  =0.
\end{array}
\right.  \label{elliptic}%
\end{equation}
For this aim we will use the properties of the fixed points of the standard
Chafee-Infante equation. In order to do that, for any $d\geq0$ we will study
the following boundary-value problem%
\begin{equation}
\left\{
\begin{array}
[c]{c}%
-a(d)\dfrac{d^{2}u}{dx^{2}}=\lambda f(u),\ 0<x<1,\\
u\left(  0\right)  =u\left(  1\right)  =0,
\end{array}
\right.  \label{FixedChaffee}%
\end{equation}
as it is obvious that $u\left(  \text{\textperiodcentered}\right)  $ is
solution to problem (\ref{elliptic}) if and only if $u\left(
\text{\textperiodcentered}\right)  $ is a solution to problem
(\ref{FixedChaffee}) with $d=\Vert u\Vert_{H_{0}^{1}}^{2}.$

\subsection{Dependence on the parameters of the fixed points for the
Chafee-Infante equation}

Denoting $\widetilde{\lambda}=\dfrac{\lambda}{a\left(  d\right)  }$ problem
(\ref{FixedChaffee}) becomes%
\begin{equation}
\left\{
\begin{array}
[c]{c}%
-\dfrac{d^{2}u}{dx^{2}}=\widetilde{\lambda}f(u),\ 0<x<1,\\
u\left(  0\right)  =u\left(  1\right)  =0.
\end{array}
\right.  \label{Chafee2}%
\end{equation}
Assuming conditions (A1)-(A5), it is known \cite{CabCarvMarVal} that if
$n^{2}\pi^{2}<\widetilde{\lambda}\leq\left(  n+1\right)  ^{2}\pi^{2}$, then
this problem has exactly $2n+1$ solutions, denoted by $v_{0}\equiv
0,\ v_{1}^{\pm},...,v_{n}^{\pm}.$ The function $v_{k}^{\pm}$ has $k+1$ simple
zeros in $[0,1].$

We need to study the dependence of the norm of these fixed points on the
parameter $\widetilde{\lambda}$. First, we will show that the $H^{1}$-norm of
the fixed points of problem (\ref{Chafee2}) is strictly increasing with
respect to the parameter $\widetilde{\lambda}.$

\begin{lemma}
\label{Increasing}Assume conditions (A1)-(A5). Let $v_{1}=v_{k,\lambda_{1}%
}^{+},\ v_{2}=v_{k,\lambda_{2}}^{+}$ with $k^{2}\pi^{2}<\lambda_{1}%
<\lambda_{2}$. Then $\left\Vert v_{1}\right\Vert _{H_{0}^{1}}<\left\Vert
v_{2}\right\Vert _{H_{0}^{1}}.$
\end{lemma}

\begin{proof}
We consider the equivalent norm in $H_{0}^{1}\left(  \Omega\right)  $ given by
$\left\Vert v^{\prime}\right\Vert _{L^{2}}$. The fixed points are the
solutions of the initial value problem%
\begin{equation}
\left\{
\begin{array}
[c]{c}%
\dfrac{d^{2}u}{dx^{2}}+\widetilde{\lambda}f(u)=0,\\
u(0)=0,\ u^{\prime}(0)=v_{0}%
\end{array}
\right.  \label{IV}%
\end{equation}
such that $u\left(  1\right)  =0$. The solutions of (\ref{IV}) satisfy the
relation%
\begin{equation}
\frac{(u^{\prime}(x))^{2}}{2}+\widetilde{\lambda}F(u(x))=\widetilde{\lambda
}E,\ 0\leq x\leq1, \label{Level}%
\end{equation}
for some constant $E\geq0$. Denote $u_{\widetilde{\lambda}}%
=v_{k,\widetilde{\lambda}}^{+}$. By Theorem 7 in \cite{CabCarvMarVal} we have
that $u_{\widetilde{\lambda}}$ is associated with a unique value $E=E_{k}%
^{+}(\widetilde{\lambda})>0.$ Moreover, $E_{k}^{+}(\widetilde{\lambda})$ is a
solution of one of the following equations:%
\begin{align}
m\tau_{+}^{\widetilde{\lambda}}(E)+(m-1)\tau_{-}^{\widetilde{\lambda}}(E)  &
=\frac{1}{\sqrt{2}},\nonumber\\
m\tau_{-}^{\widetilde{\lambda}}(E)+(m-1)\tau_{+}^{\widetilde{\lambda}}(E)  &
=\frac{1}{\sqrt{2}},\nonumber\\
m\tau_{+}^{\widetilde{\lambda}}(E)+m\tau_{-}^{\widetilde{\lambda}}(E)  &
=\frac{1}{\sqrt{2}}, \label{Roots}%
\end{align}
where either $k=2m-1$ or $k=2m$ and
\begin{equation}
\tau_{+}^{\widetilde{\lambda}}(E)=\widetilde{\lambda}^{-1/2}\int_{0}%
^{U_{+}(E)}\!\!(E-F(u))^{{-1}/{2}}\ du,\ \label{Tao+}%
\end{equation}%
\begin{equation}
\tau_{-}^{\widetilde{\lambda}}(E)=\widetilde{\lambda}^{-1/2}\int_{U_{-}%
(E)}^{0}(E-F(u))^{{-1}/{2}}\ du, \label{Tao-}%
\end{equation}
being $U_{+}(E)$ ($U_{-}(E)$) the positive (negative) inverse of $F$ at $E$.
It is obvious\ that for $E$ fixed the functions $\tau_{+}^{\widetilde{\lambda
}}(E)$, $\tau_{-}^{\widetilde{\lambda}}(E)$ are strictly decreasing with
respect to $\widetilde{\lambda}$. Then from (\ref{Roots}) we deduce that the
root $E_{k}^{+}(\widetilde{\lambda})$ is strictly increasing with respect to
$\widetilde{\lambda}$. Thus, If $\lambda_{1}<\lambda_{2}$, we have
\begin{equation}
\sqrt{2\lambda_{1}(E_{k}^{+}(\lambda_{1})-F(u))}<\sqrt{2\lambda_{2}(E_{k}%
^{+}(\lambda_{2})-F(u))},\quad U^{-}(E_{k}^{+}(\lambda_{1}))\leq u\leq
U^{+}(E_{k}^{+}(\lambda_{1})). \label{INeqSQRT}%
\end{equation}
We will prove now that $\Vert u_{\widetilde{\lambda}}^{\prime}\Vert_{L^{2}}$
is strictly increasing in $\widetilde{\lambda}$.

The function $u_{\widetilde{\lambda}}$ has $k+1$ simple zeros in $[0,1]$ and
$u_{\widetilde{\lambda}}$ is positive in the first subinterval. Let
$T_{+}(E_{k}^{+}(\lambda))$ be the $x$-time necessary to go from the initial
condition $u_{\lambda}(0)=0$ to the point where $u_{\lambda}^{\prime}%
(T_{+}(E_{k}^{+}(\lambda)))=0$. Then the length of the first subinterval is
$2T_{+}(E_{k}^{+}(\lambda))$ \cite{CabCarvMarVal}. By (\ref{Level}),
\[
(u_{\widetilde{\lambda}}^{\prime}(x))^{2}=\sqrt{2\widetilde{\lambda}}%
\sqrt{E_{k}^{+}(\widetilde{\lambda})-F(u_{\widetilde{\lambda}}(x))}%
u_{\widetilde{\lambda}}^{\prime}(x),
\]
so we have
\[
\int_{0}^{T_{+}(E_{k}^{+}(\widetilde{\lambda}))}(u_{\widetilde{\lambda}%
}^{\prime}(x))^{2}dx=\int_{0}^{T_{+}(E_{k}^{+}(\widetilde{\lambda}))}%
\sqrt{2\widetilde{\lambda}}\sqrt{E_{k}^{+}(\widetilde{\lambda}%
)-F(u_{\widetilde{\lambda}}(x))}u_{\widetilde{\lambda}}^{\prime}(x)dx.
\]
By the change of variable $v=u_{\widetilde{\lambda}}(x)$ we obtain%
\[
\int_{0}^{T_{+}(E_{k}^{+}(\widetilde{\lambda}))}(u_{\widetilde{\lambda}%
}^{\prime}(x))^{2}dx=\int_{0}^{U^{+}(E_{k}^{+}(\widetilde{\lambda}))}%
\sqrt{2\widetilde{\lambda}}\sqrt{E_{k}^{+}(\widetilde{\lambda})-F(v)}%
dv=g(\widetilde{\lambda}).
\]
Since $\widetilde{\lambda}\mapsto U^{+}(E_{k}^{+}(\widetilde{\lambda}))$ is
strictly increasing and using (\ref{INeqSQRT}), we conclude that the function
$g(\widetilde{\lambda})$ is strictly increasing. Hence, putting$\ x_{1}%
(\widetilde{\lambda})=2T_{+}(E_{k}^{+}(\widetilde{\lambda}))$ we obtain that
the norm of $u_{\widetilde{\lambda}}$ in the first subinterval, $\left\Vert
u_{\widetilde{\lambda}}^{\prime}\right\Vert _{L^{2}(0,x_{1}(\widetilde{\lambda
}))}$, is strictly increasing. Arguing in the same way in the other
subintervals we obtain that $\widetilde{\lambda}\mapsto\Vert
u_{\widetilde{\lambda}}^{\prime}\Vert_{L^{2}}$ is strictly increasing.
\end{proof}

\bigskip

Let us prove the same result but with respect to the norm $\left\Vert
u_{\widetilde{\lambda}}\right\Vert _{L^{p}}$ with $p\geq1.$

\begin{lemma}
\label{Increasing2}Assume conditions (A1)-(A5) and let $f$ be odd. Let
$v_{1}=v_{k,\lambda_{1}}^{+},\ v_{2}=v_{k,\lambda_{2}}^{+}$ with $k^{2}\pi
^{2}<\lambda_{1}<\lambda_{2}$. Then $\left\Vert v_{1}\right\Vert _{L^{p}%
}<\left\Vert v_{2}\right\Vert _{L^{p}}$ for any $p\geq1.$
\end{lemma}

\begin{proof}
As in the previous lemma, denote $u_{\widetilde{\lambda}}%
=v_{k,\widetilde{\lambda}}^{+}$. The function $u_{\widetilde{\lambda}}$ has
$k+1$ zeros in $[0,1]$ at the points $0<x_{1}<x_{2}<...<x_{k-1}<1$. When $f$
is odd, by symmetry, the length of all subintervals has to be the same, so
$x_{j}=\frac{j}{k}$ regardless the value of $\widetilde{\lambda}.$

We shall prove that in the first subinterval we have that $u_{\lambda_{1}%
}\left(  x\right)  <u_{\lambda_{2}}\left(  x\right)  ,$ for all $x\in\left(
0,\frac{1}{k}\right)  $. By (\ref{Level}) for $x\in\lbrack0,\frac{1}{2k}]$ we
have%
\[
x=\int_{0}^{x}ds=\int_{0}^{u_{\widetilde{\lambda}}\left(  x\right)  }\frac
{du}{\sqrt{2\widetilde{\lambda}\left(  E_{k}^{+}\left(  \widetilde{\lambda
}\right)  -F\left(  u\right)  \right)  }},
\]
so (\ref{INeqSQRT}) yields%
\begin{align*}
x  &  =\int_{0}^{u_{\lambda_{2}}\left(  x\right)  }\frac{du}{\sqrt
{2\lambda_{2}\left(  E_{k}^{+}\left(  \lambda_{2}\right)  -F\left(  u\right)
\right)  }}=\int_{0}^{u_{\lambda_{1}}\left(  x\right)  }\frac{du}%
{\sqrt{2\lambda_{1}\left(  E_{k}^{+}\left(  \lambda_{1}\right)  -F\left(
u\right)  \right)  }}\\
&  >\int_{0}^{u_{\lambda_{1}}\left(  x\right)  }\frac{du}{\sqrt{2\lambda
_{2}\left(  E_{k}^{+}\left(  \lambda_{2}\right)  -F\left(  u\right)  \right)
}},\text{ if }x\in(0,\frac{1}{2k}].
\end{align*}
Thus, $u_{\lambda_{1}}\left(  x\right)  <u_{\lambda_{2}}\left(  x\right)  ,$
for all $x\in(0,\frac{1}{2k}]$. By symmetry we obtain that the inequality is
true in $\left(  0,\frac{1}{k}\right)  $.

Repeating the same argument in the other subintervals we get that%
\[
\left\vert u_{\lambda_{1}}\left(  x\right)  \right\vert <\left\vert
u_{\lambda_{2}}\left(  x\right)  \right\vert \text{ for all }x\in\left(
0,1\right)  ,\ x\not =\frac{j}{k},\ j=1,...k-1.
\]
This implies that $\left\Vert u_{\lambda_{1}}\right\Vert _{L^{p}}<\left\Vert
u_{\lambda_{2}}\right\Vert _{L^{p}}$ for any $p\geq1.$
\end{proof}

\begin{remark}
The statements in Lemmas \ref{Increasing}-\ref{Increasing2} are also true for
$v_{k,\widetilde{\lambda}}^{-},$ because $v_{k,\widetilde{\lambda}}^{-}\left(
x\right)  =v_{k,\widetilde{\lambda}}^{+}\left(  1-x\right)  $, so the
$H_{0}^{1}$ and $L^{p}$ norms of $v_{k,\widetilde{\lambda}}^{-}$ and
$v_{k,\widetilde{\lambda}}^{+}$ are the same.
\end{remark}

\subsection{Nonlocal fixed points}

Although in this paper we are mainly interested in problem (\ref{problem1}),
we will study the existence of stationary points for an elliptic problem with
a more general nonlocal term than in (\ref{elliptic}). Namely, let us consider
the following problem:%
\begin{equation}
\left\{
\begin{array}
[c]{c}%
-a\left(  l(u)\right)  u_{xx}=\lambda f\left(  u\right)  ,\ 0<x<1,\\
u\left(  0\right)  =u(1)=0,
\end{array}
\right.  \label{FixedNonlocal}%
\end{equation}
where%
\[
l\left(  u\right)  =\left\Vert u\right\Vert _{H_{0}^{1}}^{r}\text{ or
}\left\Vert u\right\Vert _{L^{p}}^{r},\ p\geq1\text{, }r>0.
\]

Let%
\[
d_{k}=\sup\{d:\lambda>a\left(  \overline{d}\right)  \pi^{2}k^{2}\text{
}\forall\overline{d}\leq d\}.
\]
Then for any $d<d_{k}$ there exists the fixed point $u_{k}^{d}$ of
(\ref{FixedChaffee}), where $u_{k}^{d}$ is either equal to $u_{k}^{+}$ or
$u_{k}^{-}.$

It is obvious that any solution of (\ref{FixedNonlocal}) is a solution of
(\ref{FixedChaffee})\ with $d=l\left(  u\right)  .$ Therefore, all the
solutions to problem (\ref{FixedNonlocal}) have to be solutions $u_{k}^{d}$ to
problem (\ref{FixedChaffee}) for a suitable $d.$

\begin{theorem}
\label{FixedPointsExist}Assume conditions (A1)-(A6) and, additionally, that%
\begin{equation}
a\left(  0\right)  \pi^{2}k^{2}<\lambda. \label{CondLambda}%
\end{equation}
Then:

\begin{itemize}
\item For any $1\leq j\leq k$ there exists $d_{j}^{\ast}<d_{k}$ such that
$u_{j}^{d_{j}^{\ast}}$ is a fixed point of problem (\ref{FixedNonlocal}).

\item If $\lambda\leq a\left(  0\right)  \pi^{2}\left(  k+1\right)  ^{2}$ and
$a(0)=\min_{s\geq0}\{a\left(  s\right)  \}$, there are no fixed points for
$j>k.$

\item If $N\geq k$ is the first integer such that $\lambda\leq\inf_{s\geq
0}\{a\left(  s\right)  \pi^{2}\left(  N+1\right)  ^{2}\}$, there are no fixed
points for $j>N.$

\item If $l\left(  u\right)  =\left\Vert u\right\Vert _{H_{0}^{1}}^{r}$,
$\lambda\leq a\left(  0\right)  \pi^{2}\left(  k+1\right)  ^{2}$ and $a$ is
non-decreasing, there are exactly $2k+1$ solutions to problem
(\ref{FixedNonlocal}):\ $0,\ u_{1,d_{1}^{\ast}}^{\pm},...,u_{k,d_{k}^{\ast}%
}^{\pm}.$

\item If $l\left(  u\right)  =\left\Vert u\right\Vert _{L^{p}}^{r}$,
$\lambda\leq a\left(  0\right)  \pi^{2}\left(  k+1\right)  ^{2}$, $f$ is odd
and $a$ is non-decreasing, there are exactly $2k+1$ solutions to problem
(\ref{FixedNonlocal}):\ $0,\ u_{1,d_{1}^{\ast}}^{\pm},...,u_{k,d_{k}^{\ast}%
}^{\pm}.$
\end{itemize}
\end{theorem}

\begin{proof}
For the first statement, it is enough to prove the result for $j=k$. By
condition (\ref{CondLambda}) we have that $d_{k}\in(0,+\infty]$.

Consider first the case where $d_{k}$ is finite. We need to obtain the
existence of $d_{k}^{\ast}<d_{k}$ such that $l\left(  u_{k}^{d_{k}^{\ast}%
}\right)  =d_{k}^{\ast}$. When $d=0$ it is clear that $l\left(  u_{k}%
^{0}\right)  >0$. Also, we know that $l\left(  u_{k}^{d_{k}}\right)  =0$.
Multiplying (\ref{FixedChaffee}) by $u_{k}^{d}$ and using (\ref{condf}), (A6)
and the Poincar\'{e} inequality we obtain%
\[
\left\Vert \left(  u_{k}^{d}\right)  ^{\prime}\right\Vert _{L^{2}}^{2}%
\leq\frac{\lambda}{a\left(  d\right)  }\left(  f\left(  u_{k}^{d}\right)
,u_{k}^{d}\right)  \leq\frac{\lambda}{m}\left(  m_{\varepsilon}+\varepsilon
\left\Vert u_{k}^{d}\right\Vert _{L^{2}}^{2}\right)  \leq K_{1}+\frac{1}%
{2}\left\Vert \left(  u_{k}^{d}\right)  ^{\prime}\right\Vert _{L^{2}}^{2},
\]
so, by using the embedding $H_{0}^{1}\left(  \Omega\right)  \subset L^{\infty
}\left(  \Omega\right)  ,$ $l\left(  u_{k}^{d}\right)  $ is bounded in $d$.
This implies that the function $g\left(  d\right)  =l\left(  u_{k}^{d}\right)
$ has to intersect the line $y\left(  d\right)  =d$ at some point $d_{k}%
^{\ast}$. It remains to check that $d_{k}^{\ast}<d_{k}$. For this aim we prove
first that $u_{k}^{d}\underset{d\rightarrow d_{k}}{\rightarrow}0$ strongly in
$H_{0}^{1}\left(  \Omega\right)  $. Indeed, as $u_{k}^{d}$ is bounded in
$H_{0}^{1}\left(  \Omega\right)  $, there exist $v$ and a sequence
$\{u_{k}^{d_{j}}\}$ such that $u_{k}^{d_{j}}\rightarrow v$ in $L^{2}\left(
\Omega\right)  $. The embedding $H_{0}^{1}\left(  \Omega\right)  \subset
C\left(  [0,1]\right)  $ and the continuity of the function $f\left(
u\right)  $ imply that $\{f(u_{k}^{d_{j}})\}$ is bounded in $C([0,1])$, so
from%
\[
\left\Vert \left(  u_{k}^{d_{j}}\right)  ^{\prime\prime}\right\Vert _{L^{2}%
}\leq\frac{\lambda}{a\left(  d_{j}\right)  }\left\Vert f\left(  u_{k}^{d_{j}%
}\right)  \right\Vert _{L^{2}}\leq\frac{\lambda}{m}\left\Vert f\left(
u_{k}^{d_{j}}\right)  \right\Vert _{L^{2}}\leq C
\]
we deduce that $\{u_{k}^{d_{j}}\}$ is bounded in $H^{2}\left(  \Omega\right)
$. Hence, $u_{k}^{d_{j}}\rightarrow v$ in $H_{0}^{1}\left(  \Omega\right)  $
and $C^{1}([0,1])$. Also, $f(u_{k}^{d_{j}})\rightarrow f(v)$ in $C\left(
[0,1]\right)  $. Therefore, for any $\psi\in H_{0}^{1}\left(  \Omega\right)  $
we have that%
\[%
\begin{array}
[c]{ccc}%
\left(  \left(  u_{k}^{d_{j}}\right)  ^{\prime},\psi^{\prime}\right)   & = &
\frac{\lambda}{a\left(  d_{j}\right)  }\left(  f\left(  u_{k}^{d_{j}}\right)
,\psi\right)  \\
\downarrow &  & \downarrow\\
\left(  v^{\prime},\psi^{\prime}\right)   & = & \frac{\lambda}{a\left(
d_{k}\right)  }\left(  f\left(  v\right)  ,\psi\right)  ,
\end{array}
\]
which implies that $v$ is a solution to problem (\ref{FixedChaffee}) with
$d=d_{k}$. But from $u_{k}^{d_{j}}\rightarrow v$ in $C^{1}([0,1])$ it follows
that $v$ cannot be a point with less than $k+1$ simple zeros in $[0,1]$ and
then $\lambda/a\left(  d_{k}\right)  =k^{2}\pi^{2}$ implies that $v\equiv0$.
As the limit is the same for every converging subsequence, $u_{k}%
^{d}\underset{d\rightarrow d_{k}}{\rightarrow}0$ strongly in $H_{0}^{1}\left(
\Omega\right)  $. Thus, $d_{k}>0$ and $\lim_{d\rightarrow d_{k}}\ \left\Vert
\left(  u_{k}^{d}\right)  ^{\prime}\right\Vert _{L^{2}}=0$ imply that
$d_{k}^{\ast}<d_{k}$.

Second, let $d_{k}=+\infty$. Then the existence of $d_{k}^{\ast}<+\infty$
follows by the same argument as before.

The second and third statements are a consequence of%
\[
\lambda\leq a\left(  0\right)  \pi^{2}\left(  k+1\right)  ^{2}\leq a\left(
d\right)  \pi^{2}\left(  k+1\right)  ^{2}\text{ for any }d\geq0
\]
and%
\[
\lambda\leq\inf_{s\geq0}\{a\left(  s\right)  \}\pi^{2}\left(  N+1\right)
^{2}\leq a\left(  d\right)  \pi^{2}\left(  N+1\right)  ^{2}\text{ for any
}d\geq0\text{,}%
\]
respectively, because in such a case for problem (\ref{FixedChaffee}) the
fixed points $v_{j}^{\pm}$, $j>k$ (respectively $j>N$), do not exist.

The last two statements are a consequence of the first two statements and of
the fact that the points of intersection of the functions $g\left(  d\right)
=l\left(  u_{k}^{d}\right)  $ and $y\left(  d\right)  =d$ has to be unique,
because if $a$ is non-decreasing, then $g(d)$ is non-increasing by Lemmas
\ref{Increasing} and \ref{Increasing2}.
\end{proof}

\bigskip

In view of this theorem, we have exactly the same equilibria and bifurcations
as in the classical Chafee-Infante equation (see \cite{Henry85},
\cite{CabCarvMarVal}) when the function $a(d)$ is non-decreasing, because in
this case in view of the monotone dependence between the functions $a(d)$ and
$g(d)$, there is only one intersection point of the function $g\left(
d\right)  $ with the bisector, as it is shown in Figure \ref{fig1}. This
follows from the fact that $g(d)-d$ is strictly decreasing, but there may be
weaker conditions on $a\left(  \text{\textperiodcentered}\right)  $ that would
lead $g(d)-d$ to be strictly decreasing.

When the function $a\left(  \text{\textperiodcentered}\right)  $ is not
assumed to be monotone, an interesting situation appears. More precisely, it
is possible to have more than two equilibria with the same number of zeros. If
$l\left(  u\right)  =\left\Vert u\right\Vert _{H_{0}^{1}}^{2}$, for the
equilibria with $k+1$ zeros in $[0,1]$ this happens when the equation
\begin{equation}
d=\int_{0}^{1}\left\vert \frac{du_{k}^{d}(x)}{dx}\right\vert ^{2}%
dx=g(d)\label{condexist}%
\end{equation}
has more than one solution. For instance, if $a(0)=a(\overline{d})$ for some
$0<\bar{d}<g(0)$, then $g(0)=g(\overline{d})$. Assuming that there are
$0<d_{k}^{1}<d_{k}^{2}<$ $\overline{d}$ such that $a(d_{k}^{2})=a(d_{k}%
^{1})=\frac{\lambda}{\pi^{2}k^{2}},$ there must exist $0<d_{1}^{\ast}%
<d_{k}^{1}<d_{k}^{2}<d_{2}^{\ast}<\overline{d}$ such that $g(d_{i}^{\ast
})=d_{i}^{\ast}$. Now, by the argument in Theorem \ref{FixedPointsExist},
there must exist a $d_{3}^{\ast}>\overline{d}$ such that $g(d_{3}^{\ast
})=d_{3}^{\ast}$, obtaining six fixed points with $k+1$ zeros in $[0,1]$. This
situation is shown in Figure \ref{fig2}, where $d_{1}^{\ast},d_{2}^{\ast}$ and
$d_{3}^{\ast}$ are solutions of (\ref{condexist}), that is, there are three
intersection points with the bisector. We notice that when $a(d)>\lambda
/(\pi^{2}k^{2})$, the function $g(d)$ is not defined since the condition for
such equilibria to exist is not satisfied, but we can make this function
continuous by putting $g\left(  d\right)  =0$ whenever $a(d)\geq\lambda
/(\pi^{2}k^{2})$. This procedure establishes that, having fixed a natural
number $k$, for any $j\in\mathbb{N}$ we may construct $a\left(
\text{\textperiodcentered}\right)  $ in such a way that we have $2(2j+1)$
equilibria with $k+1$ zeros in $[0,1]$.

At least there is always one intersection point with the bisector, but the
function $g(d)$ could be even tangent to the bisector at some point or not cut
it again.

\begin{figure}[h]
\begin{minipage}[b]{0.5\linewidth}
\centering
\includegraphics[width=7cm, height=7.75cm]{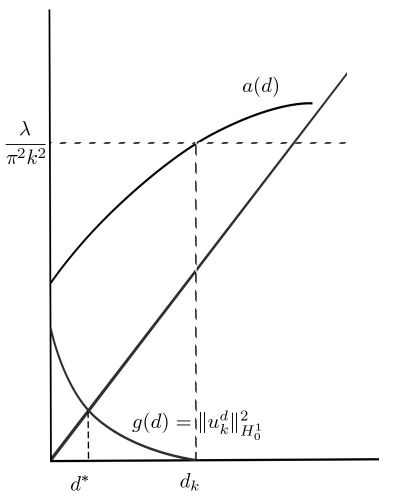}
\caption{$a(d)$ non-decreasing}
\label{fig1}
\end{minipage}
\hspace{0.1cm} \begin{minipage}[b]{0.5\linewidth}
\centering
\includegraphics[width=7.33cm, height=7.75cm]{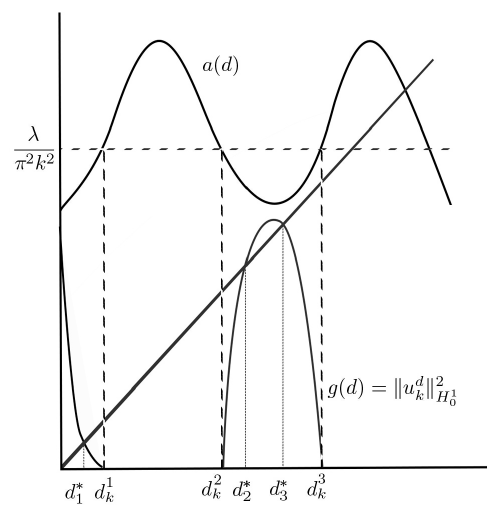}
\caption{$a(d)$ whatever}
\label{fig2}
\end{minipage}
\end{figure}

\subsection{Lap number and some forbidden connections}

With Theorem \ref{FixedPointsExist} at hand we can improve the description of
the global attractor given in Theorem \ref{theorem12}.

Under conditions (A1)-(A6), (A8) and $h\equiv0$, if
\begin{equation}
a\left(  0\right)  \pi^{2}n^{2}<\lambda\leq a\left(  0\right)  \pi^{2}\left(
n+1\right)  ^{2}\label{lambda}%
\end{equation}
then problem (\ref{problem1})\ possesses exactly $2n+1$ fixed points:\ $v_{0}%
=0,\ u_{1,d_{1}^{\ast}}^{\pm},...,u_{n,d_{n}^{\ast}}^{\pm}.$

Let $\phi$ be a bounded complete trajectory. We know by Theorem
\ref{theorem12} that%
\[
dist_{H_{0}^{1}}(\phi\left(  t\right)  ,\mathfrak{R})\rightarrow0\text{, as
}t\rightarrow\pm\infty.
\]
As the number of fixed points is finite, we will prove that in fact the
solution has to converge to one fixed point forwards and backwards. We recall
the omega and alpha limit sets of $\phi$, given by%
\begin{align*}
\omega\left(  \phi\right)   &  =\{y:\exists t_{n}\rightarrow+\infty\text{ such
that }\phi\left(  t_{n}\right)  \rightarrow y\},\\
\alpha\left(  \phi\right)   &  =\{y:\exists t_{n}\rightarrow-\infty\text{ such
that }\phi\left(  t_{n}\right)  \rightarrow y\},
\end{align*}
are non-empty, compact and connected \cite[Lemma 3.4 and Proposition
4.1]{Ball}. Also, $dist_{H_{0}^{1}}\left(  \phi\left(  t\right)
,\omega\left(  \phi\right)  \right)  \underset{t\rightarrow+\infty
}{\rightarrow}0,\ dist_{H_{0}^{1}}\left(  \phi\left(  t\right)  ,\alpha\left(
\phi\right)  \right)  \underset{t\rightarrow-\infty}{\rightarrow}0$. Since
$\omega\left(  \phi\right)  ,\alpha\left(  \phi\right)  \subset\mathfrak{R}$
and $\mathfrak{R}$ is finite, the only possibility is that $\omega\left(
\phi\right)  =z_{1}\in\mathfrak{R}$, $\alpha\left(  \phi\right)  =z_{2}%
\in\mathfrak{R}$.

Thus, we have established the following result.

\begin{theorem}
\label{Structure2}Let assume conditions (A1)-(A6), (A8), (\ref{lambda}) and
$h\equiv0$. Then
\[
\mathcal{A=\cup}_{k=0}^{2n+1}M^{+}\left(  v_{k}\right)  =\mathcal{\cup}%
_{k=0}^{2n+1}M^{-}\left(  v_{k}\right)  ,
\]
where $n$ is given in (\ref{lambda}) and $v_{0}=0$, $v_{1}=u_{1,d_{1}^{\ast}%
}^{+},\ v_{2}=u_{1,d_{1}^{\ast}}^{-},...$

In other words, the global attractor $\mathcal{A}$ consists of the set of
stationary points $\mathfrak{R}$ (which has $2n+1$ elements) and the bounded
complete trajectories that connect them (the heteroclinic connections).
\end{theorem}

\begin{remark}
\label{Homoclinic}As the Lyapunov function (\ref{Lyapunov}) is strictly
decreasing along a trajectory $\phi$ which is not a fixed point, then there
cannot exist homoclinic connections for any fixed point. This implies in
particular that if $n=0$, then $\mathcal{A}=\{0\}.$
\end{remark}

\begin{remark}
If we use condition (A7) instead of (A8), then we cannot guarantee that the
number of fixed points is finite. But if we suppose that this is the case,
then the result remains valid. In this situation, there could be more than two
fixed points with the same number of zeros.
\end{remark}

\begin{lemma}
\label{NoConnection1}Let assume conditions (A1)-(A6), $h\equiv0$ and either
(A7) or (A8). Let $u_{k,d_{k}^{\ast}}^{+},u_{k,d_{k}^{\ast}}^{-}$ be a pair of
fixed points corresponding to the same value $d_{k}^{\ast}$. Then there cannot
be an heteroclinic connection between them.
\end{lemma}

\begin{proof}
The function $v\left(  x\right)  =u_{k,d_{k}^{\ast}}^{+}\left(  1-x\right)  $
is a fixed point corresponding to $d_{k}^{\ast}$ as%
\[
-\frac{\partial^{2}v}{\partial x^{2}}\left(  x\right)  =-\frac{\partial
^{2}u_{k,d_{k}^{\ast}}^{+}}{\partial x^{2}}\left(  1-x\right)  =\frac{\lambda
}{a\left(  d_{k}^{\ast}\right)  }f\left(  u_{k,d_{k}^{\ast}}^{+}\left(
1-x\right)  \right)  =\frac{\lambda}{a\left(  d_{k}^{\ast}\right)  }f\left(
v\left(  x\right)  \right)  ,
\]
so $u_{k,d_{k}^{\ast}}^{-}\left(  x\right)  =v\left(  x\right)  =u_{k,d_{k}%
^{\ast}}^{+}\left(  1-x\right)  $. The equalities%
\[
\int_{0}^{1}\left(  \frac{\partial u_{k,d_{k}^{\ast}}^{-}}{\partial x}\left(
x\right)  \right)  ^{2}dx=\int_{0}^{1}\left(  \frac{\partial u_{k,d_{k}^{\ast
}}^{+}}{\partial x}\left(  1-x\right)  \right)  ^{2}dx=\int_{0}^{1}\left(
\frac{\partial u_{k,d_{k}^{\ast}}^{+}}{\partial x}\left(  y\right)  \right)
^{2}dy,
\]%
\[
\int_{0}^{1}\int_{0}^{u_{d_{k}^{\ast}}^{-}\left(  x\right)  }f\left(
s\right)  dsdx=\int_{0}^{1}\int_{0}^{u_{d_{k}^{\ast}}^{+}\left(  1-x\right)
}f\left(  s\right)  dsdx=\int_{0}^{1}\int_{0}^{u_{d_{k}^{\ast}}^{+}\left(
y\right)  }f\left(  s\right)  dsdy
\]
imply that $E(u_{k,d_{k}^{\ast}}^{-})=E\left(  u_{k,d_{k}^{\ast}}^{+}\right)
$, where $E$ is the Lyapunov function (\ref{Lyapunov}). Since this function is
strictly decreasing along a trajectory $\phi$ which is not a fixed point,
there cannot exist a heteroclinic connection between these two points.
\end{proof}

\begin{remark}
In the case where condition (A7) is assumed, there could be more than two
equilibria with $k+1$ zeros in $[0,1]$. In this case there could exist
connections between fixed points with different values of the constant $d$.
\end{remark}

\bigskip

Using the concept of lap number of the solutions we can discard some more
heteroclinic connections.

We consider the function $w\left(  t\right)  =u(\alpha^{-1}\left(  t\right)
)$, which is a strong solution to problem (\ref{w}). For any strong solution
$u\left(  \text{\textperiodcentered}\right)  $ conditions (A1), (A3), (A6) and
$u\in C([0,+\infty),H_{0}^{1}(\Omega))$ imply that the function%
\[
r\left(  t,x\right)  =\frac{\lambda}{a(\Vert w\left(  t\right)  \Vert
_{H_{0}^{1}}^{2})}\frac{f\left(  w\left(  t,x\right)  \right)  }{w\left(
t,x\right)  }%
\]
is continuous and $w\left(  \text{\textperiodcentered}\right)  $ is a solution
of the linear equation
\begin{equation}
\dfrac{\partial w}{\partial t}-\dfrac{\partial^{2}w}{\partial x^{2}}=r\left(
t,x\right)  w.\label{Linear}%
\end{equation}

Thus, by Theorem \ref{LapNumber} in the Appendix (see also Theorem C in
\cite{Angenent}) the number of zeros of $w\left(  t\right)  $ in $[0,1]$ is a
nonincreasing function of $t$. Since $\alpha^{-1}\left(  t\right)  $ is an
increasing function of time, the result is also true for the solution
$u\left(  \text{\textperiodcentered}\right)  $. Making use of this property we
will prove the following result.

\begin{lemma}
\label{NoConnection2}Let assume conditions (A1)-(A6), $h\equiv0$ and either
(A7) or (A8). Then if $n>k$, there cannot exist a connection from the fixed
point $u_{k,d_{k}^{\ast}}^{\pm}$ to the fixed point $u_{n,d_{n}^{\ast}}^{\pm}%
$, that is, there cannot exist a bounded complete trajectory $\phi$ such that%
\[
\phi\left(  t\right)  \rightarrow u_{n,d_{n}^{\ast}}^{\pm}\text{ as
}t\rightarrow+\infty,\ \phi\left(  t\right)  \rightarrow u_{k,d_{k}^{\ast}%
}^{\pm}\text{ as }t\rightarrow-\infty.
\]

\end{lemma}

\begin{proof}
By contradiction assume that such complete trajectory exists. Denote by
$l\left(  z\right)  $ the number of zeros of $z$ in $[0,1]$. Using the
compactness of the attractor in $C^{1}([0,1])$ (see Corollary \ref{CompactC1})
we obtain that
\begin{align*}
\phi\left(  t\right)   &  \rightarrow u_{n,d_{n}^{\ast}}^{\pm}\text{ in }%
C^{1}([0,1])\text{ as }t\rightarrow+\infty,\\
\phi\left(  t\right)   &  \rightarrow u_{k,d_{k}^{\ast}}^{\pm}\text{ in }%
C^{1}([0,1])\text{ as }t\rightarrow-\infty.
\end{align*}
Then, as the zeros are simple, we can choose $t_{1}>0$ large enough such that
$l\left(  \phi\left(  -t_{1}\right)  \right)  =l\left(  u_{k,d_{k}^{\ast}%
}^{\pm}\right)  =k+1.$ Put $u\left(  t\right)  =\phi\left(  t-t_{1}\right)  $,
which is a strong solution of (\ref{problem1}). Now we choose $t_{2}>0$ such
that $l\left(  u\left(  t_{2}\right)  \right)  =l\left(  u_{n,d_{n}^{\ast}%
}^{\pm}\right)  =n+1$. Then $l\left(  u\left(  0\right)  \right)  =k+1$ and
$l\left(  u(t_{2})\right)  =n+1>k+1$. This contradicts the fact that the
number of zeros of $u\left(  t\right)  $ is non-increasing.
\end{proof}

\section{Morse decomposition\label{DynamicalGrad}}

In this section we study in more detail the structure of the global attactor
in the case where the function $f$ is odd. More precisely, we obtain that the
m-semiflow $G$ is dynamically gradient, which is equivalent to saying that
there is a Morse decomposition of the attractor \cite{CostaValero16}, and
study the stability of the fixed points.

\subsection{Aproximations}

We consider now the situation when conditions (A1)-(A6), $h=0$ and either (A7)
or (A8) are satisfied and, moreover, the function $f$ is odd.

In this section we consider the following problems:
\begin{equation}
\left\{
\begin{array}
[c]{l}%
\dfrac{\partial u}{\partial t}-a(\Vert u\Vert_{H_{0}^{1}}^{2})\dfrac
{\partial^{2}u}{\partial x^{2}}=\lambda f_{\varepsilon_{n}}(u),\quad
t>0,\ x\in(0,1),\\
u(t,0)=0,\ u(t,1)=0,\\
u(0,x)=u_{0}(x),
\end{array}
\right.  \label{problemaprox}%
\end{equation}
where the function $f_{\varepsilon_{n}}$ is defined below and $\varepsilon
_{n}\rightarrow0$, as $n\rightarrow\infty$.

Let $\rho_{\varepsilon_{n}}(\cdot)$ be a mollifier in $\mathbb{R}$. We define
the function $f^{\varepsilon_{n}}(u)=\int_{\mathbb{R}}\rho_{\varepsilon_{n}%
}(s)f(u-s)ds$. It is well known that $f^{\varepsilon_{n}}(\cdot)\in C^{\infty
}(\mathbb{R})$ and that for any compact subset $A\subset\mathbb{R}$ we have
$f^{\varepsilon_{n}}\rightarrow f$ uniformly on $A$. It is clear that for
$u>\varepsilon_{n}$ the function $f^{\varepsilon_{n}}\left(  u\right)  $ is
strictly concave.

We need the approximation to fulfil (A2)-(A3). For that end, we consider the
approximation except on the interval $[-\varepsilon_{n},\varepsilon_{n}]$, for
any $\varepsilon_{n}>0$. There exists a polynomial of sixth degree $p(x)$ such
that
\[%
\begin{split}
&  p(0)=0,\quad\quad\quad p(\varepsilon_{n})=h(\varepsilon_{n}),\\
&  p^{\prime}(0)=1,\quad\quad\quad p^{\prime}(\varepsilon_{n})=h^{\prime
}(\varepsilon_{n}),\\
&  p^{\prime\prime}(0)=0,\quad\quad\quad p^{\prime\prime}(\varepsilon
_{n})=h^{\prime\prime}(\varepsilon_{n}),\\
&  p^{\prime\prime\prime}(0)=-1.
\end{split}
\]
We choose $\gamma>0$ such that $p^{\prime\prime}\left(  s\right)  <0$ for all
$s\in(0,\gamma]$. We can assume that $\varepsilon_{n}<\gamma$ for all $n$.

Thus, by construction the function
\begin{equation}
f_{\varepsilon_{n}}(x)=\left\{
\begin{array}
[c]{lcc}%
-f^{\varepsilon_{n}}(-x) & if & x<-\varepsilon_{n},\\
-p(-x) & if & -\varepsilon_{n}\leq x\leq0,\\
p(x) & if & 0\leq x\leq\varepsilon_{n},\\
f^{\varepsilon_{n}}(x) & if & x>\varepsilon_{n}%
\end{array}
\right.  \label{faproxim}%
\end{equation}
approximates the function $f$ uniformly in compact sets, that is, for any
$[-M,M]$ and $\delta>0$ there exists $n_{0}(M,\delta)\in\mathbb{N}$ such that
\begin{equation}
|f(x)-f_{\varepsilon_{n}}(x)|<\delta,\quad\text{ for all }n\geq n_{0}%
,\ x\in\lbrack-M,M].\label{uniformconverg}%
\end{equation}
Also, it satisfies the following properties:

\begin{enumerate}
\item[(B1)] $f_{\varepsilon_{n}}\in C^{2}(\mathbb{R});$

\item[(B2)] $f_{\varepsilon_{n}}(0)=0;$

\item[(B3)] $f_{\varepsilon_{n}}^{\prime}\left(  0\right)  =1;$

\item[(B4)] $f_{\varepsilon_{n}}$ is strictly concave if $u>0$ and strictly
convex if $u<0$;

\item[(B5)] $f_{\varepsilon_{n}}$ is odd.
\end{enumerate}

\begin{lemma}
\label{ApproxA5}Let $f$ satisfy (A5). Then the functions $f_{\varepsilon_{n}}$
satisfy condition (A5) and (\ref{condf}) with independent constants of
$\varepsilon_{n}$.
\end{lemma}

\begin{proof}
We assume without loss of generality that $\varepsilon_{n}<1$. In order to
check (\ref{A51})-(\ref{A52}) we only need to consider $u$ outside the
interval $[-1,1]$, because the sequence $\{f_{\varepsilon_{n}}\}$ is uniformly
bounded in any compact set of $\mathbb{R}$. Then for $u\not \in \lbrack-1,1]$
the H\"{o}lder inequality and $\int_{\mathbb{R}}\rho_{\varepsilon_{n}}\left(
s\right)  ds=1$ give%
\begin{align*}
\left\vert f_{\varepsilon_{n}}\left(  u\right)  \right\vert  &  =\left\vert
\int_{\mathbb{R}}f\left(  u-s\right)  \rho_{\varepsilon_{n}}\left(  s\right)
ds\right\vert \leq\int_{\mathbb{R}}\left\vert f\left(  u-s\right)  \right\vert
\rho_{\varepsilon_{n}}\left(  s\right)  ds\\
&  \leq\int_{\mathbb{R}}\left(  C_{1}+C_{2}|u-s|^{p-1}\right)  \rho
_{\varepsilon_{n}}\left(  s\right)  ds\\
&  \leq C_{1}+C_{2}2^{p-2}\left(  \int_{-\varepsilon_{n}}^{\varepsilon_{n}%
}\left(  \left\vert u\right\vert ^{p-1}+\left\vert s\right\vert ^{p-1}\right)
\rho_{\varepsilon_{n}}\left(  s\right)  ds\right)  \\
&  \leq\widetilde{C}_{1}+\widetilde{C}_{2}\left\vert u\right\vert ^{p-1}.
\end{align*}

If $f$ satisfies (\ref{A52}), then%
\begin{align*}
f_{\varepsilon_{n}}\left(  u\right)  u &  =\int_{\mathbb{R}}f\left(
u-s\right)  \left(  u-s\right)  \rho_{\varepsilon_{n}}\left(  s\right)
ds+\int_{\mathbb{R}}f\left(  u-s\right)  s\rho_{\varepsilon_{n}}\left(
s\right)  ds\\
&  \leq\int_{\mathbb{R}}\left(  C_{3}-C_{4}|u-s|^{p}\right)  \rho
_{\varepsilon_{n}}\left(  s\right)  ds+\int_{\mathbb{R}}\left(  C_{1}%
+C_{2}|u-s|^{p-1}\right)  s\rho_{\varepsilon_{n}}\left(  s\right)  ds\\
&  \leq K_{1}-C_{4}\int_{\mathbb{R}}\left(  2^{1-p}\left\vert u\right\vert
^{p}-\left\vert s\right\vert ^{p}\right)  \rho_{\varepsilon_{n}}\left(
s\right)  ds\\
&  +C_{2}2^{p-2}\int_{\mathbb{R}}\left(  \left\vert u\right\vert
^{p-1}+\left\vert s\right\vert ^{p-1}\right)  s\rho_{\varepsilon_{n}}\left(
s\right)  ds\\
&  \leq\widetilde{C}_{3}-\widetilde{C}_{4}|u|^{p},
\end{align*}
where we have used $\left\vert u\right\vert ^{p}\leq2^{p-1}\left(  \left\vert
s^{p}\right\vert +\left\vert u-s\right\vert ^{p}\right)  $ and the Young inequality.

For (\ref{condf}) we put in the above inequality $p=2$, $C_{3}=m_{\varepsilon
},\ C_{4}=-\varepsilon$ and obtain%
\[
f_{\varepsilon_{n}}\left(  u\right)  u\leq\widetilde{m}_{\varepsilon
}+\varepsilon u^{2},
\]
which obviously implies (\ref{A53}).
\end{proof}

\bigskip

Our next aim is to focus on the convergence of solutions of the approximations.

\begin{theorem}
\label{solconver}Let conditions (A1)-(A6), $h=0$ and either (A7) or (A8) be
satisfied and let, moreover, the function $f$ be odd. If $u_{\varepsilon
_{n},0}\rightarrow u_{0}$ in $H_{0}^{1}(\Omega)$ as $\varepsilon
_{n}\rightarrow0$, then for any sequence of solutions of (\ref{problemaprox})
$u_{\varepsilon_{n}}(\cdot)$ with $u_{\varepsilon_{n}}(0)=u_{\varepsilon
_{n},0}$ there exists a subsequence of $\varepsilon_{n}$ such that
$u_{\varepsilon_{n}}$ converges to some strong solution $u\left(
\text{\textperiodcentered}\right)  $ of (\ref{problem1}) in the space
$C([0,T],H_{0}^{1}(\Omega))$, for any $T>0$.
\end{theorem}

\begin{proof}
Using (\ref{equality}) and (\ref{FDeriv}) we can repeat the same lines of the
proof of Theorem \ref{existence} and obtain the existence of a function
$u\left(  \text{\textperiodcentered}\right)  $ and a subsequence of
$u_{\varepsilon_{n}}$ such that%
\[
u_{\varepsilon_{n}}\overset{\ast}{\rightharpoonup}u\text{ in }L^{\infty
}(0,T;H_{0}^{1}(\Omega)),
\]%
\[
u_{\varepsilon_{n}}\rightharpoonup u\text{ in }L^{2}(0,T;D(A)),
\]%
\[
\frac{du_{\varepsilon_{n}}}{dt}{\rightharpoonup}\frac{du}{dt}\text{ in }%
L^{2}(0,T;L^{2}(\Omega)),
\]%
\[
u_{\varepsilon_{n}}\rightarrow u\text{ in }C([0,T];L^{2}(\Omega)),
\]%
\[
u_{\varepsilon_{n}}\rightarrow u\text{ in }L^{2}(0,T;H_{0}^{1}(\Omega)),
\]%
\[
f_{\varepsilon_{n}}(u_{n_{\varepsilon}})\overset{\ast}{\rightharpoonup
}f(u)\text{ in }L^{\infty}(0,T;L^{\infty}(\Omega)),
\]%
\[
a(\Vert u_{\varepsilon_{n}}\Vert_{H_{0}^{1}}^{2})\Delta u_{\varepsilon_{n}%
}\rightharpoonup a(\Vert u\Vert_{H_{0}^{1}}^{2})\Delta u\text{ in }%
L^{2}(0,T;L^{2}(\Omega)).
\]
Also, in the same way we prove that $u\left(  \text{\textperiodcentered
}\right)  $ is a strong solution to problem (\ref{problem1}) such that
$u\left(  0\right)  =u_{0}.$

The uniform estimate in the space $H_{0}^{1}\left(  \Omega\right)  $ implies
also that if $t_{n}\rightarrow t_{0}$, then $u_{\varepsilon_{n}}\left(
t_{n}\right)  \rightharpoonup u\left(  t_{0}\right)  $ in $H_{0}^{1}(\Omega)$.
We need to prove that this convergence is in fact strong, proving then the
convergence in $C([0,T],H_{0}^{1}(\Omega))$ for any $T>0$.

In the same way as in the proof of Lemma \ref{lemma4} we deduce that for some
$C>0$ the functions $Q_{n}(t)=A(\Vert u_{\varepsilon_{n}}(t)\Vert_{H_{0}^{1}%
}^{2})-2Ct,$ $Q(t)=A(\Vert u(t)\Vert_{H_{0}^{1}}^{2})-2Ct$ are continuous and
non-increasing in $[0,T].$ Moreover, $Q_{n}(t)\rightarrow Q(t)$ for a.e.
$t\in(0,T).$ Let first $t_{0}>\dot{0}$. Then we take $0<t_{j}<t_{0}$ such that
$t_{j}\rightarrow t_{0}$ and $Q_{n}(t_{j})\rightarrow Q(t_{j})$ for all $j.$
Then
\[
Q_{n}(t_{n})-Q(t_{0})\leq Q_{n}(t_{j})-Q(t_{0})\leq|Q_{n}(t_{j})-Q(t_{j}%
)|+|Q(t_{j})-Q(t_{0})|\text{ for }t_{n}>t_{j}.
\]
For any $\delta>0$ there exist $j(\delta)$ and $N(j(\delta))$ such that
$Q_{n}(t_{n})-Q(t_{0})\leq\delta$ if $n\geq N,$ so $\limsup Q_{n}(t_{n})\leq
Q(t_{0}).$ Hence, a contradiction argument using the continuity of $A\left(
s\right)  $ shows that $\limsup\Vert u_{\varepsilon_{n}}(t_{n})\Vert
_{H_{0}^{1}}^{2}\leq\Vert u(t_{0})\Vert_{H_{0}^{1}}^{2}.$ This, together with
$\liminf\Vert u_{\varepsilon_{n}}(t_{n})\Vert_{H_{0}^{1}}^{2}\geq\Vert
u(t_{0})\Vert_{H_{0}^{1}}^{2},$ implies that $\Vert u_{\varepsilon_{n}}%
(t_{n})\Vert_{H_{0}^{1}}^{2}\rightarrow\Vert u(t_{0})\Vert_{H_{0}^{1}}^{2},$
so that $u_{\varepsilon_{n}}(t_{n})\rightarrow u(t_{0})$ strongly in
$H_{0}^{1}(\Omega).$ For the case when $t_{0}=0$ we use the same argument as
in Lemma \ref{lemma4}.
\end{proof}

\bigskip

We denote by $\mathcal{A}_{\varepsilon_{n}}$ the global attractor for the
semiflow $G_{\varepsilon_{n}}$ corresponding to problem (\ref{problemaprox}).

\begin{lemma}
\label{comtrajbound}Assume the condition of Theorem \ref{solconver}. Then
$\cup_{_{n\in\mathbb{N}}}\mathcal{A}_{\varepsilon_{n}}$ is bounded in
$H_{0}^{1}(\Omega)$. Hence, the set $\overline{\cup_{n\in\mathbb{N}%
}\mathcal{A}_{\varepsilon_{n}}}$ is compact in $L^{2}(\Omega)$.
\end{lemma}

\begin{proof}
By Lemma \ref{ApproxA5} inequality (\ref{condf}) is satisfied for any $n$ with
constants which are independent of $\varepsilon_{n},$ so inequality
(\ref{Inequ}) holds true with constants independent of $\varepsilon_{n}$.
Thus, there a exists a common absorbing ball $B_{0}$ in $L^{2}\left(
\Omega\right)  $ (with radius $K>0$) for problems (\ref{problemaprox}).
Further, by repeating the same steps as in Proposition \ref{absorbingset} we
obtain a common absorbing ball in $H_{0}^{1}\left(  \Omega\right)  $ (with
radius $\widetilde{K}>0$) as by Lemma \ref{ApproxA5} the constants which are
involved are independent of $\varepsilon_{n}$. Thus, $\left\Vert y\right\Vert
_{H_{0}^{1}}\leq\widetilde{K}$ for any $y\in\cup_{_{n\in\mathbb{N}}%
}\mathcal{A}_{\varepsilon_{n}}.$
\end{proof}

\begin{lemma}
Assume the condition of Theorem \ref{solconver}. Then $\cup_{_{n\in\mathbb{N}%
}}\mathcal{A}_{\varepsilon_{n}}$ is bounded in $V^{2r}$ for any $0\leq r<1$.
Hence, $\overline{\cup_{_{n\in\mathbb{N}}}\mathcal{A}_{\varepsilon_{n}}}$ is
compact in $V^{2r}$ and $C^{1}([0,1])$.
\end{lemma}

\begin{proof}
Using Lemma \ref{comtrajbound} we obtain the boundedness of $\cup
_{_{n\in\mathbb{N}}}\mathcal{A}_{\varepsilon_{n}}$ in $V^{2r}$ by repeating
the same lines in Lemma \ref{CompactVr}. The rest of the proof follows from
the compact embedding $V^{\alpha}\subset V^{\beta}$, $\alpha>\beta,$ and the
continuous embedding $V^{2r}\subset C^{1}([0,1])$ if $r>\frac{3}{4}.$
\end{proof}

\begin{corollary}
\label{relcompact}Assume the condition of Theorem \ref{solconver}. Then any
sequence $\xi_{n}\in\mathcal{A}_{\varepsilon_{n}}$ with $\varepsilon
_{n}\rightarrow0$ is relatively compact in $C^{1}([0,1])$.
\end{corollary}

\begin{lemma}
\label{remark28}Assume the condition of Theorem \ref{solconver}. Then up to a
subsequence any bounded complete trajectory $u_{\varepsilon_{n}}$ of
(\ref{problemaprox}) converges to a bounded complete trajectory $u$ of
(\ref{problem1}) in $C([-T,T],H_{0}^{1}(\Omega))$ for any $T>0$. On top of
that, if $y_{n}\in\mathcal{A}_{\varepsilon_{n}}$, then passing to a
subsequence $y_{n}\rightarrow y\in\mathcal{A}$ in $H_{0}^{1}\left(
\Omega\right)  .$ Hence,
\begin{equation}
dist_{H_{0}^{1}}\left(  \mathcal{A}_{\varepsilon_{n}},\mathcal{A}\right)
\rightarrow0\text{ as }n\rightarrow\infty. \label{USC}%
\end{equation}

\end{lemma}

\begin{proof}
Let fix $T>0$. By Corollary \ref{relcompact} $u_{\varepsilon_{n}%
}(-T)\rightarrow y$ in $H_{0}^{1}(\Omega)$ up to a subsequence. Theorem
\ref{solconver} implies that $u_{\varepsilon_{n}}$ converges in
$C([-T,T],H_{0}^{1}(\Omega))$ to some solution $u$ of (\ref{problem1}). If we
choose successive subsequences for $-2T,-3T\ldots$ and apply the standard
diagonal procedure, we obtain that a subsequence $u_{\varepsilon_{n}}$
converges to a complete trajectory $u$ of (\ref{problem1}) in $C([-T,T],H_{0}%
^{1}(\Omega))$ for any $T>0$. Finally, from Lemma \ref{comtrajbound} this
trajectory is bounded.

If $y_{n}\in\mathcal{A}_{\varepsilon_{n}},$ by Corollary \ref{relcompact} we
can extract a subsequence converging to some $y$. If we take a sequence of
bounded complete trajectories $\phi_{n}\left(  \text{\textperiodcentered
}\right)  $ of (\ref{problemaprox}) such that $\phi_{n}\left(  0\right)
=y_{n}$, then by the previous result it converges in $C([-T,T],H_{0}%
^{1}(\Omega))$ to some bounded complete trajectory $\phi\left(
\text{\textperiodcentered}\right)  $ of (\ref{problem1}), so $y\in
\mathcal{A}.$

Finally, if (\ref{USC}) was not true, there would exist $\delta>0$ and a
sequence $y_{n}\in\mathcal{A}_{\varepsilon_{n}}$ such that $dist_{H_{0}^{1}%
}(y,\mathcal{A})>\delta$. But passing to a subsequence $y_{n}\rightarrow
y\in\mathcal{A}$, which is a contradiction.
\end{proof}

\begin{lemma}
\label{ConvergTao}Assume the conditions of Theorem \ref{solconver}. Let
$\tau_{\pm}^{d_{n},\varepsilon_{n}}$ be the functions (\ref{Tao+}%
)-(\ref{Tao-}) for problem (\ref{FixedChaffee}) but replacing $f$ by
$f_{\varepsilon_{n}}$ and $d$ by $d_{n}$. Let $d_{n},E_{n}\rightarrow0$ as
$n\rightarrow\infty$. Then
\[
\lim_{n\rightarrow\infty}\tau_{\pm}^{d_{n},\varepsilon_{n}}(E_{n})=\frac
{\sqrt{a\left(  0\right)  }\pi}{\sqrt{2\lambda}}.
\]

\end{lemma}

\begin{proof}
Let us consider $f_{d_{n},\varepsilon_{n}}(u)=\frac{\lambda f_{\varepsilon
_{n}}(u)}{a(d_{n})}$. In view of property $(B4)$ and (\ref{uniformconverg}),
since $f_{\varepsilon_{n}}^{\prime}(0)=f^{\prime}(0)=1$ and $f_{\varepsilon
_{n}}(0)=f(0)=0,$ given $\gamma\in(0,1)$ there exists $\delta>0$ (independent
of $\varepsilon_{n}$) such that
\begin{equation}%
\begin{array}
[c]{c}%
(1-\gamma)u\leq f_{\varepsilon_{n}}(u)\leq(1+\gamma)u,\quad\text{ for any
}u\in(0,\delta).\\
\frac{1}{1+\gamma}\leq\frac{u}{f_{\varepsilon_{n}}(u)}\leq\frac{1}{1-\gamma
},\quad\text{ for any }u\in(0,\delta).
\end{array}
\label{continu1}%
\end{equation}

The sequence $\mathcal{F}_{\varepsilon_{n}}\left(  \text{\textperiodcentered
}\right)  $ converges uniformly to $\mathcal{F}\left(
\text{\textperiodcentered}\right)  $ in compact sets. Moreover, as $U_{+}(E)$
is continuous and using \cite[p. 60]{souza}, given $\delta>0$, there exists
$\eta>0$ such that $U_{+}^{\varepsilon_{n}}(E)\leq\delta$ for any $0<E\leq
\eta$. Now, if we integrate the first inequality in (\ref{continu1}) between
$0$ and $u$ we obtain
\[
\frac{1}{2}(1-\gamma)u^{2}\leq\mathcal{F}_{\varepsilon_{n}}(u)\leq\frac{1}%
{2}(1+\gamma)u^{2},\quad\text{ for any }0\leq u\leq\delta.
\]
Using the change of variable $E_{n}y^{2}=\mathcal{F}_{\varepsilon_{n}}(u)$, we
have
\[%
\begin{array}
[c]{c}%
\left(  \frac{1-\gamma}{2E_{n}}\right)  ^{{1}/{2}}u\leq y\leq\left(
\frac{1+\gamma}{2E_{n}}\right)  ^{{1}/{2}}u,\quad\text{ if }0<E_{n}\leq
\eta,\ 0\leq y\leq1.
\end{array}
\]
Dividing the previous expression by $\sqrt{\frac{\lambda}{a\left(
d_{n}\right)  }}f_{d_{n},\varepsilon_{n}}(u)$ and using (\ref{continu1}) we
obtain
\[%
\begin{array}
[c]{c}%
\left(  \frac{a(d_{n})(1-\gamma)}{2\lambda E_{n}(1+\gamma)^{2}}\right)
^{{1}/{2}}\leq\frac{\sqrt{a\left(  d_{n}\right)  }y}{\sqrt{\lambda}%
f_{d_{n},\varepsilon_{n}}(u)}\leq\left(  \frac{a(d_{n})(1+\gamma)}{2\lambda
E_{n}(1-\gamma)^{2}}\right)  ^{{1}/{2}}\!\!\!,\!\!\text{if }0<E_{n}\leq
\eta,\ 0\leq y\leq1.
\end{array}
\]
Now if we multiply by $2\sqrt{E_{n}}(1-y^{2})^{-\frac{1}{2}}$ and integrate
from $0$ to $1$, we get
\[%
\begin{array}
[c]{c}%
\!\!\pi\left(  \frac{a(d_{n})(1-\gamma)}{2\lambda(1+\gamma)^{2}}\right)
^{{1}/{2}}\leq\tau_{+}^{\varepsilon_{n}}(E_{n})\leq\pi\left(  \frac
{a(d_{n})(1+\gamma)}{2\lambda(1-\gamma)^{2}}\right)  ^{{1}/{2}},\!\quad\text{
if }0<E_{n}\leq\eta.
\end{array}
\]
Then the theorem follows as $a\left(  d_{n}\right)  \rightarrow a\left(
0\right)  $ when $n\rightarrow\infty$. The proof for $\tau_{-}^{\varepsilon
_{n}}$ is analogous.
\end{proof}

\bigskip

Under the conditions of Theorem \ref{solconver}, if (A8) is satisfied and
\begin{equation}
a\left(  0\right)  \pi^{2}k^{2}<\lambda\leq a\left(  0\right)  \pi^{2}\left(
k+1\right)  ^{2},\ k\in\mathbb{Z},\ k\geq0, \label{CondLambda2}%
\end{equation}
holds, then by Theorem \ref{FixedPointsExist} problem (\ref{problemaprox}) has
exactly $2k+1$ fixed points (denoted by $v_{0}=0,\ \ v_{1,d_{1}^{\varepsilon
_{n}}}^{\pm},...,v_{k,d_{k}^{\varepsilon_{n}}}^{\pm}$) and $v_{m,d_{m}%
^{\varepsilon_{n}}}^{\pm}$ has $m+1$ zeros in $[0,1]$ for each $1\leq m\leq
k$. The same is valid for problem (\ref{problem1}) and we denote the $2k+1$
fixed points by $v_{0}=0,\ u_{1,d_{1}^{\ast}}^{\pm},...,u_{k,d_{k}^{\ast}%
}^{\pm}$.

\begin{lemma}
\label{lemaconvergenciaA0} Assume the conditions of Theorem \ref{solconver},
(A8) and (\ref{CondLambda2}). Let $m\in\mathbb{N}$, $1\leq m\leq k,$ be fixed.
Then $v_{m,d_{m}^{\varepsilon_{n}}}^{+}$ (resp. $v_{m,d_{m}^{\varepsilon_{n}}%
}^{-})$ do not converge to $0$ in $H_{0}^{1}(\Omega)$ as $\varepsilon
_{n}\rightarrow0$.
\end{lemma}

\begin{proof}
Assume that $v_{m,d_{m}^{\varepsilon_{n}}}^{+}\rightarrow0$ in $H_{0}%
^{1}\left(  0,1\right)  $. Then it converges to $0$ in $C\left(  [0,1]\right)
$ and the equality%
\[
-\frac{d^{2}v_{m,d_{m}^{\varepsilon_{n}}}^{+}}{dx^{2}}\left(  x\right)
=\frac{\lambda f_{\varepsilon_{n}}\left(  v_{m,d_{m}^{\varepsilon_{n}}}%
^{+}\left(  x\right)  \right)  }{a(d_{m}^{\varepsilon_{n}})}%
\]
implies that $v_{m,d_{m}^{\varepsilon_{n}}}^{+}\rightarrow0$ in $C^{2}\left(
[0,1]\right)  $. In particular, $\dfrac{dv_{m,d_{m}^{\varepsilon_{n}}}^{+}%
}{dx}\left(  0\right)  \rightarrow0$ and $d_{m}^{\varepsilon_{n}}=\left\Vert
v_{m,d_{m}^{\varepsilon_{n}}}^{+}\right\Vert _{H_{0}^{1}}^{2}\rightarrow0$.
The value $E_{n}$ corresponding to the fixed point $v_{m,d_{m}^{\varepsilon
_{n}}}^{+}$ is equal to $\dfrac{a\left(  d_{m}^{\varepsilon_{n}}\right)
}{2\lambda}\dfrac{dv_{m,d_{m}^{\varepsilon_{n}}}^{+}}{dx}\left(  0\right)  $,
so $E_{n}\rightarrow0$. We will show that this is not possible. We know by
Lemma \ref{ConvergTao} that%
\[
\lim_{n\rightarrow\infty}\tau_{\pm}^{d_{m}^{\varepsilon_{n}},\varepsilon_{n}%
}\left(  E_{n}\right)  =\frac{\pi\sqrt{a(0)}}{\sqrt{2\lambda}}.
\]
Also, since $v_{m,d_{m}^{\varepsilon_{n}}}^{+}$ is a fixed point with
$d=d_{m}^{\varepsilon_{n}}$ one of the following conditions has to be
satisfied (see (\ref{Roots})):
\begin{equation}
j\tau_{+}^{d_{m}^{\varepsilon_{n}},\varepsilon_{n}}\left(  E_{n}\right)
+\left(  j-1\right)  \tau_{-}^{d_{m}^{\varepsilon_{n}},\varepsilon_{n}}\left(
E_{n}\right)  =\left(  \frac{1}{2}\right)  ^{\frac{1}{2}}, \label{63}%
\end{equation}%
\begin{equation}
j\tau_{-}^{d_{m}^{\varepsilon_{n}},\varepsilon_{n}}\left(  E_{n}\right)
+\left(  j-1\right)  \tau_{+}^{d_{m}^{\varepsilon_{n}},\varepsilon_{n}}\left(
E_{n}\right)  =\left(  \frac{1}{2}\right)  ^{\frac{1}{2}},\text{ if }m=2j-1
\label{64}%
\end{equation}%
\begin{equation}
j\tau_{+}^{d_{m}^{\varepsilon_{n}},\varepsilon_{n}}(E_{n})+j\tau_{-}%
^{d_{m}^{\varepsilon_{n}},\varepsilon_{n}}(E_{n})=\left(  \frac{1}{2}\right)
^{\frac{1}{2}},\text{ if }m=2j. \label{655}%
\end{equation}
Since $E_{n}\rightarrow0$ and $\lambda>k^{2}\pi^{2}a(0)\geq m^{2}\pi^{2}a(0)$,
there exists $\varepsilon_{n_{0}}$ such that
\[
\tau_{\pm}^{d_{m}^{\varepsilon_{n_{0}}},\varepsilon_{n_{0}}}(E_{n_{0}}%
)<\frac{1}{\sqrt{2}m}.
\]
Hence, neither of (\ref{63})-(\ref{655}) is possible.
\end{proof}

\bigskip

\begin{lemma}
\label{convergfixedpoint}Assume the conditions of Theorem \ref{solconver},
(A8) and (\ref{CondLambda2}). Let $m\in\mathbb{N}$, $1\leq m\leq k,$ be fixed.
Then $v_{m,d_{m}^{\varepsilon_{n}}}^{+}$ (resp. $v_{m,d_{m}^{\varepsilon_{n}}%
}^{-})$ converges to $u_{m,d_{m}^{\ast}}^{+}$ in $H_{0}^{1}(\Omega)$ (resp.
$u_{m,d_{m}^{\ast}}^{-}$) as $\varepsilon_{n}\rightarrow0$.
\end{lemma}

\begin{proof}
We consider $v_{m,d_{m}^{\varepsilon_{n}}}^{+}.$ In view of Corollary
\ref{relcompact}, $v_{m,d_{m}^{\varepsilon_{n}}}^{+}$ is relatively compact in
$C^{1}\left(  [0,1]\right)  ,$ so up to a subsequence $v_{m,d_{m}%
^{\varepsilon_{n}}}^{+}\rightarrow v$ strongly in $C^{1}([0,1])$ and
$d_{m}^{\varepsilon_{n}}\rightarrow d^{\ast}=\left\Vert v\right\Vert
_{H_{0}^{1}}^{2}$. The proof will be finished if we prove that $v=u_{m,d_{m}%
^{\ast}}^{+}$. We observe that since in such a case every subsequence would
have the same limit, the whole sequence would converge to $u_{m,d_{m}^{\ast}%
}^{+}$.

In view of (\ref{uniformconverg}) $f_{\varepsilon_{n}}(v_{m,d_{m}%
^{\varepsilon_{n}}}^{+})$ converges to $f\left(  v\right)  $ in $C([0,1]).$ It
follows that
\[
-\frac{\partial^{2}v}{\partial x^{2}}=\frac{\lambda f\left(  v\right)
}{a(\left\Vert v\right\Vert _{H_{0}^{1}}^{2})}%
\]
and $v$ is a solution of (\ref{elliptic}), so $v$ is a fixed point of
(\ref{problem1}). We need to prove that $v=u_{m,d_{m}^{\ast}}^{+}$. By Lemma
\ref{lemaconvergenciaA0} $v\not =0$, and then $v=u_{j,d_{j}^{\ast}}^{\pm}$ for
some $1\leq j\leq k$. Since $u_{j,d_{j}^{\ast}}^{\pm}$ has $j+1$ simple zeros,
the convergence $v_{m,d_{m}^{\varepsilon_{n}}}^{+}\rightarrow u_{j,d_{j}%
^{\ast}}^{\pm}$ in $C^{1}([0,1])$ implies that $v_{m,d_{m}^{\varepsilon_{n}}%
}^{+}$ has $j+1$ zeros for $n\geq N$. But $v_{m,d_{m}^{\varepsilon_{n}}}^{+}$
possesses $m+1$ zeros in $[0,1]$. Thus, $m=j$.

For the sequence $v_{m,d_{m}^{\varepsilon_{n}}}^{-}$ the proof is analogous.
\end{proof}

\subsection{Instability}

We will prove that the fixed points $0$ and $u_{k,d_{k}^{\ast}}^{\pm}$,
$k\geq2,$ are unstable under some additional assumptions on the functions $f$
and $a$. For this aim we need to use the approximative problems
(\ref{problemaprox}).

\begin{theorem}
\label{Unstable}Assume that the conditions (A1)-(A8), $h=0$,
(\ref{CondLambda2}) with $k\geq1$ are satisfied and let, moreover, the
function $f\left(  \text{\textperiodcentered}\right)  $ be odd and $a\left(
\text{\textperiodcentered}\right)  $ be globally Lipschitz continuous. Then
the equilibria $v_{0}=0$ and $u_{j,d_{j}^{\ast}}^{\pm}$, $2\leq j\leq k$ (if
$k\geq2$), are unstable.
\end{theorem}

\begin{remark}
The condition that $a\left(  \text{\textperiodcentered}\right)  $ is globally
Lipschitz continuous could be dropped, as we can replace $a\left(
\text{\textperiodcentered}\right)  $ in (\ref{problemaprox}) by a sequence
$a_{\varepsilon_{n}}\left(  \text{\textperiodcentered}\right)  $ of globally
Lipschitz continuous functions.
\end{remark}

\begin{proof}
Problem (\ref{problemaprox}) generates a single-valued semigroup
$\{T_{\varepsilon_{n}}(t);t\geq0\}$ with a finite number of fixed
points:\ $v_{0}=0,\ v_{1,d_{1}^{\varepsilon_{n}}}^{\pm},...,v_{k,d_{k}%
^{\varepsilon_{n}}}^{\pm}$ \cite{carvalestef}. We know by Theorems 3.5 and 3.6
in \cite{carvalestef} that for any $v_{j,d_{j}^{\varepsilon_{n}}}^{+}$ with
$j\geq2$ and $v_{0}$ there exists a bounded complete trajectory
$u^{\varepsilon_{n}}$ such that
\[
u^{\varepsilon_{n}}(t)\rightarrow v_{j,d_{j}^{\varepsilon_{n}}}^{+}\quad\text{
as }t\rightarrow-\infty,\quad\text{ for }k\geq2,
\]
so $v_{0},\ v_{j,d_{j}^{\varepsilon_{n}}}^{+}$ are unstable. The same is valid
for $v_{j,d_{j}^{\varepsilon_{n}}}^{-}.$ On the other hand, by Lemma
\ref{convergfixedpoint} we have
\begin{equation}
v_{j,d_{j}^{\varepsilon_{n}}}^{\pm}\rightarrow u_{j,d_{j}^{\ast}}^{\pm},
\label{ConvergPuntoFijo}%
\end{equation}
where $u_{j,d_{j}^{\ast}}^{\pm}$ is a fixed point of problem (\ref{problem1})
with $j+1$ zeros in $[0,1]$. We prove the result for $u_{j,d_{j}^{\ast}}^{+}$.
For $u_{j,d_{j}^{\ast}}^{-}$ and $v_{0}$ the proof is the same.

By Lemma \ref{remark28} we obtain that up to a subsequence $u^{\varepsilon
_{n}}$ converges to a bounded complete trajectory $u$ of problem
(\ref{problem1}) in the space $C([-T,T],H_{0}^{1}(\Omega))$ for every $T>0$.
Thus, either $u\left(  \text{\textperiodcentered}\right)  $ is a fixed point
$v_{-1}$ or by Theorem \ref{theorem12} there exists a fixed point $v_{-1}$ of
problem (\ref{problem1}) such that
\[
u(t)\rightarrow v_{-1}\quad\text{ as }t\rightarrow-\infty\text{ in }H_{0}%
^{1}(\Omega).
\]
In the second case, if $v_{-1}=u_{j,d_{j}^{\ast}}^{+},$ the proof would be
finished, so let assume the opposite.

Assume first that either $u\left(  \text{\textperiodcentered}\right)  $ is not
a fixed point or it is a fixed point but $v_{-1}\not =u_{j,d_{j}^{\ast}}^{+}$.
We consider $r_{0}>0$ such that the neighborhood $\mathcal{O}_{2r_{0}}%
(v_{-1})$ does not contain any other fixed point of problem (\ref{problem1}).
For any $r\leq r_{0}$ we can choose $t_{r}\rightarrow-\infty$ and $n_{r}$ such
that $u^{\varepsilon_{n}}(t_{r})\in\mathcal{O}_{r}(v_{-1})$ for all $n\geq
n_{r}$. On the other hand, since $u^{\varepsilon_{n}}(t)\rightarrow
v_{j,d_{j}^{\varepsilon_{n}}}^{+}$, as $t\rightarrow-\infty,$ and
$v_{j,d_{j}^{\varepsilon_{n}}}^{+}\rightarrow u_{j,d_{j}^{\ast}}^{+}%
\not \in B_{2r_{0}}(v_{-1})$, there exists $t_{r}^{\prime}<t_{r}$ such that
\[
u^{\varepsilon_{n_{r}}}(t)\in\mathcal{O}_{r_{0}}(v_{-1})\text{ for }t\in
(t_{r}^{\prime},t_{r}],
\]%
\[
\left\Vert u^{\varepsilon_{n_{r}}}(t_{r}^{\prime})-v_{-1}\right\Vert
_{H_{0}^{1}}=r_{0}.
\]
Let first $t_{t}-t_{r}^{\prime}\rightarrow+\infty$. We define the sequence
$u_{1}^{\varepsilon_{n_{r}}}(t)=u^{\varepsilon_{n_{r}}}(t+t_{r}^{\prime}),$
which passing to a subsequence converges to a bounded complete trajectory
$\phi\left(  t\right)  $ such that $\phi\left(  t\right)  \in\mathcal{O}%
_{r_{0}}(v_{-1})$ for all $t\geq0$. As there is no other fixed point in
$\mathcal{O}_{2r_{0}}(v_{-1})$, $\phi\left(  t\right)  \rightarrow v_{-1}$ as
$t\rightarrow+\infty$. But $\left\Vert \phi\left(  0\right)  -v_{-1}%
\right\Vert =r_{0}$, so $\phi\left(  \text{\textperiodcentered}\right)  $ is
not a fixed point. Then $\phi\left(  t\right)  \rightarrow v_{-2}$ as
$t\rightarrow-\infty$, where $v_{-2}$ is a fixed point different from $v_{-1}%
$. Second, let $\left\vert t_{t}-t_{r}^{\prime}\right\vert \leq C$. Then put
$u_{1}^{\varepsilon_{n_{r}}}(t)=u^{\varepsilon_{n_{r}}}(t+t_{r}).\ $Passing to
a subsequence we have that
\begin{align*}
u_{1}^{\varepsilon_{n_{r}}}(0) &  \rightarrow v_{-1},\\
t_{r}-t_{r}^{\prime} &  \rightarrow t_{0},\ \text{as }r\rightarrow0.
\end{align*}
Also, $u_{1}^{\varepsilon_{n_{r}}}\left(  \text{\textperiodcentered}\right)  $
converges to a bounded complete trajectory $u^{1}\left(
\text{\textperiodcentered}\right)  $ of problem (\ref{problem1}) such that
$u^{1}(0)=v_{-1}.$ Let
\[
\psi_{1}(t)=\left\{
\begin{array}
[c]{c}%
u^{1}\left(  t\right)  \text{ if }t\leq0,\\
v_{-1}\text{ if }t\geq0.
\end{array}
\right.
\]
We note that $\left\Vert u^{1}(-t_{0})-v_{-1}\right\Vert _{H_{0}^{1}}=r_{0}$
implies that $u^{1}\left(  \text{\textperiodcentered}\right)  $ is not a fixed
point. Then $\psi_{1}$ is a bounded complete trajectory of problem
(\ref{problem1}) such that $\psi_{1}(t)\rightarrow v_{-2}\not =v_{-1}$ as
$t\rightarrow-\infty$. If $v_{-2}=u_{j,d_{j}^{\ast}}^{+},$ the proof is finished.

If $v_{-2}\not =u_{j,d_{j}^{\ast}}^{+}$, we continue constructing by the same
procedure a chain of connections in which the new fixed point is always
different from the previous ones, because the existence of the Lyapunov
function (\ref{Lyapunov}) avoids the existence of a cyclic chain of
connections. Since the number of fixed points is finite, at some moment we
obtain a bounded complete trajectory $\phi\left(  \text{\textperiodcentered
}\right)  $ such that $\phi\left(  t\right)  \rightarrow u_{j,d_{j}^{\ast}%
}^{+}$ as $t\rightarrow-\infty$, proving that $u_{j,d_{j}^{\ast}}^{+}$ is unstable.

Now let $u\left(  \text{\textperiodcentered}\right)  =v_{-1}=u_{j,d_{j}^{\ast
}}^{+}$. Defining the neighborhood $\mathcal{O}_{2r_{0}}(v_{-1})$ as before,
for any $r\leq r_{0}$ we can choose $n_{r}$ such that $u^{\varepsilon_{n}%
}(0)\in\mathcal{O}_{r}(v_{-1})$ for all $n\geq n_{r}$. Also, since
$u^{\varepsilon_{n}}(t)\rightarrow z_{0}^{n}$, as $t\rightarrow+\infty,$ where
$z_{0}^{n}\not =v_{j,d_{j}^{\varepsilon_{n}}}^{+}$ is a fixed point of
(\ref{problemaprox}), there exists $t_{r}>0$ such that
\[
u^{\varepsilon_{n_{r}}}(t)\in\mathcal{O}_{r_{0}}(v_{-1})\text{ for }%
t\in\lbrack0,t_{r}),
\]%
\[
\left\Vert u^{\varepsilon_{n_{r}}}(t_{r})-v_{-1}\right\Vert _{H_{0}^{1}}%
=r_{0}.
\]
The sequence $\{t_{r}\}$ cannot be bounded. Indeed, if $t_{r}\rightarrow
t_{0}$, then $u^{\varepsilon_{n_{r}}}(t_{r})\rightarrow u\left(  t_{0}\right)
=v_{-1}$, which is a contradiction with $\left\Vert u^{\varepsilon_{n_{r}}%
}(t_{0})-v_{-1}\right\Vert _{H_{0}^{1}}=r_{0}$. Then $t_{r}\rightarrow+\infty
$. We define the functions $u_{1}^{\varepsilon_{n_{r}}}(t)=u^{\varepsilon
_{n_{r}}}(t+t_{r})$, which satisfy that $u_{1}^{\varepsilon_{n_{r}}}%
(t)\in\mathcal{O}_{r_{0}}(v_{-1})$ for all $t\in\lbrack-t_{r},0)$. Passing to
a subsequence it converges to a bounded complete trajectory $\phi\left(
\text{\textperiodcentered}\right)  $ such that $\phi\left(  t\right)
\in\mathcal{O}_{r_{0}}(v_{-1})$ for all $t\leq0$. This trajectory is not a
fixed point as $\left\Vert \phi(0)-v_{-1}\right\Vert _{H_{0}^{1}}=r_{0}$ and
$\phi\left(  t\right)  \rightarrow u_{j,d_{j}^{\ast}}^{+}$ as $t\rightarrow
-\infty$, so $u_{j,d_{j}^{\ast}}^{+}$ is unstable.
\end{proof}

\bigskip

Further, we will prove that there is also a connection from $0$ to the point
$u_{k,d_{k}^{\ast}}^{\pm}$.

\begin{theorem}
\label{ConnectionsZero}Assume the conditions of Theorem \ref{Unstable}. Then
there exists a bounded complete trajectory $\phi\left(
\text{\textperiodcentered}\right)  $ such that $\phi\left(  t\right)
\underset{t\rightarrow-\infty}{\rightarrow}0$, $\phi\left(  t\right)
\underset{t\rightarrow+\infty}{\rightarrow}u_{k,d_{k}^{\ast}}^{+}$ (and the
same is valid for $u_{k,d_{k}^{\ast}}^{-}$). Thus, $E(0)=0>E(u_{k,d_{k}^{\ast
}}^{\pm}).$
\end{theorem}

\begin{proof}
We start with the case where $k=1$. We have three fixed points:
$0,\ u_{1,d_{1}^{\ast}}^{+},u_{1,d_{1}^{\ast}}^{-}$. By Theorem \ref{Unstable}
there exists a bounded complete trajectory $\phi\left(
\text{\textperiodcentered}\right)  $ such that $\phi\left(  t\right)
\underset{t\rightarrow-\infty}{\rightarrow}0$, whereas Theorem \ref{theorem12}
and Remark \ref{Homoclinic}\ imply that it has to converge forward to a fixed
point different from $0$, that is, to either $u_{1,d_{1}^{\ast}}^{+}$ or
$u_{1,d_{1}^{\ast}}^{-}$. If, for example, $\phi\left(  t\right)
\underset{t\rightarrow+\infty}{\rightarrow}u_{1,d_{1}^{\ast}}^{+}$, then as
the function $f$ is odd, $\psi\left(  t\right)  =-\phi\left(  t\right)  $ is
another bounded complete trajectory and $\psi\left(  t\right)
\underset{t\rightarrow+\infty}{\rightarrow}-u_{1,d_{1}^{\ast}}^{+}%
=u_{1,d_{1}^{\ast}}^{-}.$

Further we consider the problem%
\begin{equation}
\left\{
\begin{array}
[c]{l}%
\dfrac{\partial u}{\partial t}-a(\Vert u\Vert_{H_{0}^{1}}^{2})\dfrac
{\partial^{2}u}{\partial x^{2}}=\lambda f_{k}(u),\quad t>0,\ 0<x<\frac{1}%
{k},\\
u(t,0)=u(t,\frac{1}{k})=0,\\
u(0,x)=u_{0}(x),
\end{array}
\right.  \label{Problem_k}%
\end{equation}
where $f_{k}(u)=\sqrt{k}f\left(  u/\sqrt{k}\right)  $ satisfies (A1)-(A5). In
this problem, condition (\ref{CondLambda2}) implies that there are again three
fixed points: $0,\ u_{1,d_{1}^{\ast},\frac{1}{k}}^{+},u_{1,d_{1}^{\ast}%
,\frac{1}{k}}^{-}$. By the above argument there is a connection $\phi
_{\frac{1}{k}}\left(  \text{\textperiodcentered}\right)  $ from $0$ to
$u_{1,d_{1}^{\ast},\frac{1}{k}}^{+}$ (also to $u_{1,d_{1}^{\ast},\frac{1}{k}%
}^{-}$). Since the function $f$ is odd, $u_{k,d_{k}^{\ast}}^{+}\left(
x\right)  $ is equal to $\frac{1}{\sqrt{k}}u_{1,d_{1}^{\ast},\frac{1}{k}}%
^{+}\left(  x\right)  $ on $[0,\frac{1}{k}]$, to $-\frac{1}{\sqrt{k}%
}u_{1,d_{1}^{\ast},\frac{1}{k}}^{+}\left(  x-\frac{1}{k}\right)  $ on
$[\frac{1}{k},\frac{2}{k}]$, etc. Then the function $\phi\left(
\text{\textperiodcentered}\right)  $ such that $\phi\left(  t,x\right)
=\frac{(-1)^{j}}{\sqrt{k}}\phi_{\frac{1}{k}}\left(  t,x-\frac{j}{k}\right)  $
on $[\frac{j}{k},\frac{j+1}{k}]$, $j=0,1,...,k-1$, is a bounded complete
trajectory of problem (\ref{problem1}) which goes from $0$ to $u_{k,d_{k}%
^{\ast}}^{+}$.
\end{proof}

\begin{remark}
When $k=1$ the structure of the global attractor is the same as in the
Chafee-Infante equation.
\end{remark}

\subsection{Gradient structure}

We will obtain that the m-semiflow $G$ is dynamically gradient. Let us recall
this concept.

A weakly invariant set $M$ of $X$ is isolated if there is a neighborhood $O$
of $M$ such that $M$ is the maximal weakly invariant subset on $\mathcal{O}$.
If $M$ belongs to the global attractor $\mathcal{A}$, then it is compact
\cite[Lemma 19]{CostaValero16}. In this case, it is equivalent to use a
$\delta$-neighborhood $\mathcal{O}_{\delta}(M)=\{y\in X:dist\left(
y,M\right)  <\delta\}.$

Suppose that there is a finite disjoint family of isolated weakly invariant
sets $\mathcal{M}=\{M_{1},\ldots,M_{m}\}$ in $\mathcal{A}$, that is, for every
$j\in\{1,\ldots,n\}$ there is $\epsilon_{j}>0$ such that $M_{j}\subset
\mathcal{A}$ is the maximal weakly invariant set on $\mathcal{O}_{\epsilon
_{j}}(M_{j})$, and suppose that there exists $\delta>0$ such that
$\mathcal{O}_{\delta}(M_{i})\cap\mathcal{O}_{\delta}(M_{j})=\emptyset$, if
$i\neq j$.

\begin{definition}
We say the m-semiflow $G\colon\mathbb{R}^{+}\times X\rightarrow P(X)$ is
dynamically gradient with respect to the disjoint family of isolated weakly
invariant sets $\mathcal{M}=\{M_{1},\ldots,M_{m}\}$ if for every complete and
bounded trajectory $\psi$ of $\mathcal{R}$ we have that either $\psi
(\mathbb{R})\subset M_{j}$, for some $j\in\{1,\ldots,m\}$, or $\alpha
(\psi)\subset M_{i}$ and $\omega(\psi)\subset M_{j}$ with $1\leq j<i\leq m$.
\end{definition}

Let us consider the case when the conditions of Theorem \ref{Unstable} hold.
Then (\ref{problem1})\ possesses exactly $2k+1$ fixed points:\ $v_{0}%
=0,\ u_{1,d_{1}^{\ast}}^{\pm},...,u_{k,d_{k}^{\ast}}^{\pm}$. Also, as $f$ is
odd, $u_{j,d_{j}^{\ast}}^{+}=-u_{j,d_{j}^{\ast}}^{-}$ for any $j$. We define
the following sets:%
\begin{equation}
M_{1}=\{u_{1,d_{1}^{\ast}}^{+},u_{1,d_{1}^{\ast}}^{-}\},...,\ M_{k}%
=\{u_{k,d_{k}^{\ast}}^{+},u_{k,d_{k}^{\ast}}^{-}\},\ M_{k+1}%
=\{0\}.\label{Family}%
\end{equation}
They are weakly invariant and using Lemma \ref{NoConnection1} we deduce easily
that they are isolated. Then the family $\mathcal{M}=\{M_{1},\ldots,M_{k+1}\}$
is a finite disjoint family of isolated weakly invariant sets.

\begin{proposition}
\label{DynamGradProp}Assume the conditions of Theorem \ref{Unstable}. Then $G$
is dynamically gradient with respect to the family (\ref{Family}) after
(possibly) reordering them.
\end{proposition}

\begin{proof}
We reorder the family (\ref{Family}) in such a way that if the value of the
Lyapunov function $E$ given in (\ref{Lyapunov}) is equal to $L_{i}$ for the
set $\widetilde{M}_{i}$, then $L_{j}\leq L_{n}$ for $j<n.$ Then Theorem 25 in
\cite{CostaValero16} implies that $G$ is dynamically gradient with respect to
this family.
\end{proof}

\bigskip

We will obtain then that the fixed points $u_{1,d_{1}^{\ast}}^{+}%
,u_{1,d_{1}^{\ast}}^{-}$ are asymptotically stable. The compact set
$M\subset\mathcal{A}$ is a local attractor for $G$ in $X$ if there is
$\varepsilon>0$ such that $\omega\left(  O_{\varepsilon}(M)\right)  =M,$
where
\[
\omega\left(  B\right)  =\{y:\exists t_{n}\rightarrow+\infty,\ y_{n}\in
G(t_{n},B)\text{ such that }y_{n}\rightarrow y\}
\]
is the $\omega$-limit set of $B$. By Lemma 14 in \cite{CostaValero16} if $M$
is a local attractor in $X$, then it is stable. Thus, a local attractor is
asymptotically stable.

\begin{theorem}
\label{Stability}Assume the conditions of Theorem \ref{Unstable}. Then the
stationary points $u_{1,d_{1}^{\ast}}^{+},u_{1,d_{1}^{\ast}}^{-}$ are
asymptotically stable.
\end{theorem}

\begin{proof}
By \cite[Theorem 23 and Lemma 15]{CostaValero16} $\widetilde{M}_{1}$ is a
local attractor in $X$, so it is asymptotically stable. By Theorem
\ref{Unstable} the sets $M_{j},$ $j\geq2$, are unstable. Thus, $\widetilde{M}%
_{1}=M_{1}$. As $M_{1}$ consists of the two elements $u_{1,d_{1}^{\ast}}%
^{+},u_{1,d_{1}^{\ast}}^{-}$, which are obviously disjoint, they are
asymptotically stable as well.
\end{proof}

\bigskip

We will prove that there is a connection from $0$ to any other fixed point
$u_{j,d_{j}^{\ast}}^{\pm}.$

\begin{theorem}
\label{ConnectionFrom0}Assume the conditions of Theorem \ref{Unstable}. Then
there exists a bounded complete trajectory $\phi\left(
\text{\textperiodcentered}\right)  $ such that $\phi\left(  t\right)
\underset{t\rightarrow-\infty}{\rightarrow}0$, $\phi\left(  t\right)
\underset{t\rightarrow+\infty}{\rightarrow}u_{j,d_{j}^{\ast}}^{+}$ for all
$1\leq j\leq k$ (and the same is valid for $u_{j,d_{j}^{\ast}}^{-}$).
\end{theorem}

\begin{proof}
Let us consider problem (\ref{Problem_k}) with $k=j$. The function
$u_{1,d_{j}^{\ast},\frac{1}{j}}^{+}(x)=\sqrt{j}u_{j,d_{j}^{\ast}}^{+}(x)$,
$x\in\lbrack0,\frac{1}{j}]$, is the unique positive fixed point of problem
(\ref{Problem_k}). Let $X_{j}^{+}=\{u\in H_{0}^{1}\left(  0,\frac{1}%
{j}\right)  :u\left(  x\right)  \geq0$ $\forall x\in\lbrack0,\frac{1}{j}]\}$
be the positive cone of $H_{0}^{1}\left(  0,\frac{1}{j}\right)  $. If we
consider the restriction of the semigroup $T_{j}^{\varepsilon_{n}}\left(
\text{\textperiodcentered}\right)  $ of problem (\ref{problemaprox}) in the
interval $\left(  0,\frac{1}{j}\right)  $ to $X_{j}^{+}$, denoted by
$T_{j}^{\varepsilon_{n},+}\left(  \text{\textperiodcentered}\right)  $, then
there exists a global attractor $\mathcal{A}_{n,j}^{+}$ \cite{CarLiLuMo}.
Since $0$ and $v_{1,d_{j}^{\varepsilon_{n}},\frac{1}{j}}^{+}=\sqrt
{j}v_{j,d_{j}^{\varepsilon_{n}}}^{+}\mid_{\lbrack0,\frac{1}{j}]}$ are the
unique fixed points of $T_{j}^{\varepsilon_{n},+}$, $\mathcal{A}_{n,j}^{+}$ is
connected, $v_{1,d_{1}^{\varepsilon_{n}},\frac{1}{j}}^{+}$ is stable
\cite{carvalestef} and $\mathcal{A}_{n,j}^{+}$ consists of the fixed points
and their heteroclinic connections, there must exist a bounded complete
trajectory $\phi_{j}^{\varepsilon_{n}}\left(  \text{\textperiodcentered
}\right)  $ of $T_{j}^{\varepsilon_{n},+}$ which goes from $0$ to
$v_{1,d_{j}^{\varepsilon_{n}},\frac{1}{j}}^{+}$. By Lemma \ref{remark28} up to
a subsequence it converges to a bounded complete trajectory $\phi_{j}\left(
\text{\textperiodcentered}\right)  $ of problem (\ref{Problem_k}) with $k=j$
such that $\phi_{j}\left(  t\right)  \geq0$ for all $t\in\mathbb{R}$. Since by
Theorem \ref{Stability}\ the fixed point $u_{1,d_{j}^{\ast},\frac{1}{j}}^{+}$
is stable, the only possibility is that $\phi_{j}\left(  t\right)
\rightarrow0$, as $t\rightarrow-\infty$, $\phi_{j}\left(  t\right)
\rightarrow u_{1,d_{j}^{\ast},\frac{1}{j}}^{+}$, as $t\rightarrow+\infty$.
Then the function $\phi\left(  \text{\textperiodcentered}\right)  $ such that
$\phi\left(  t,x\right)  =\frac{\left(  -1\right)  ^{i}}{\sqrt{j}}\phi
_{j}\left(  t,x-\frac{i}{j}\right)  $ on $[\frac{i}{j},\frac{i+1}{j}]$,
$i=0,1,...,j-1$, is a bounded complete trajectory of problem (\ref{problem1})
which goes from $0$ to $u_{j,d_{j}^{\ast}}^{+}$.

For $u_{j,d_{j}^{\ast}}^{-}$, noting that $u_{j,d_{j}^{\ast}}^{-}%
=-u_{j,d_{j}^{\ast}}^{+}$, the result follows by choosing the bounded complete
trajectory $\widetilde{\phi}(t)=-\phi\left(  t\right)  $.
\end{proof}

\bigskip

As a consequence we obtain that the order of the family $\mathcal{M}$ has to
be the one given in (\ref{Family}).

\begin{theorem}
The semiflow $G$ is dynamically gradient with respect to the family
$\mathcal{M}$ in the order given in (\ref{Family}), that is, $\widetilde{M}%
_{i}=M_{i}$ for any $i$.
\end{theorem}

\begin{proof}
As by Theorem \ref{ConnectionFrom0} there is a connection from $0$ to
$u_{j,d_{j}^{\ast}}^{\pm}$, $1\leq j\leq k,$ we have proved that
$\widetilde{M}_{k+1}=\{0\}=M_{k+1}$. The fact that the order of the other sets
is the one given in (\ref{Family}) follows from Lemma \ref{NoConnection2}.
\end{proof}

\section{Appendix}

In this appendix we generalize the lap number property of solutions of linear
equations proved in \cite{Henry85} to the case when we do not have classical
solutions. For this we will use a maximum principle for non-smooth functions
from \cite{Kadlec}.

Let $\mathcal{O}$ be a region in $\mathbb{R}^{2}$ and let $\left(  t_{0}%
,x_{0}\right)  \in\mathcal{O}$ and $\rho,\sigma>0.$ We denote%
\[
Q_{\rho,\sigma}=\{(t,x):t\in(t_{0}-\sigma,t_{0}),\left\vert x-x_{0}\right\vert
<\rho\},
\]
where we assume that $t_{0},x_{0},\rho,\sigma$ are such that $\overline
{Q}_{\rho,\sigma}\subset\mathcal{O}$.

We denote by $W$ the space of all functions from $L^{2}\left(  \mathcal{O}%
\right)  $ such that%
\[
\int_{\mathcal{O}}\left(  \left\vert u\left(  t,x\right)  \right\vert
^{2}+\left\vert \frac{\partial u}{\partial x}\left(  t,x\right)  \right\vert
^{2}\right)  d\mu<+\infty.
\]

As a particular case of Theorem 6.4 in \cite{Kadlec} we obtain the following
maximum and minimum principles.

\begin{theorem}
\label{Maximum}(Maximum principle) Let $u\in W$ be such that
\begin{equation}
\frac{\partial u}{\partial t}-\frac{\partial^{2}u}{\partial x^{2}}\leq0
\label{IneqMax}%
\end{equation}
in the sense of distributions. If%
\[
\sup ess_{\left(  t,x\right)  \in Q_{\rho\nu,\sigma_{1}}}u(t,x)=M,
\]
for some $\nu$, $0<\nu<1$, and any $\sigma_{1}$, where $0<\sigma_{1}<\sigma$,
then $u\left(  t,x\right)  =M$ for a.a. $\left(  t,x\right)  \in
Q_{\rho,\sigma}.$
\end{theorem}

\begin{theorem}
\label{Minimum}(Minimum principle) Let $u\in W$ be such that
\begin{equation}
\frac{\partial u}{\partial t}-\frac{\partial^{2}u}{\partial x^{2}}\geq0
\label{IneqMin}%
\end{equation}
in the sense of distributions. If%
\[
\inf ess_{\left(  t,x\right)  \in Q_{\rho\nu,\sigma_{1}}}u(t,x)=M,
\]
for some $\nu$, $0<\nu<1$, and any $\sigma_{1}$, where $0<\sigma_{1}<\sigma$,
then $u\left(  t,x\right)  =M$ for a.a. $\left(  t,x\right)  \in
Q_{\rho,\sigma}.$
\end{theorem}

\bigskip

We are ready to prove the lap-number property, saying that the number of zeros
is a non-increasing function of time.

\begin{theorem}
\label{LapNumber}Let $r\left(  t,x\right)  $ be a continuous function and
$u\in C([t_{0},t_{1}],H_{0}^{1}\left(  \Omega\right)  )\cap L^{2}\left(
t_{0},t_{1};H^{2}\left(  \Omega\right)  \right)  $ be such that $\dfrac
{du}{dt}\in L^{2}\left(  t_{0},t_{1};L^{2}\left(  \Omega\right)  \right)  $
and satisfies the equation%
\begin{equation}
\frac{\partial u}{\partial t}-\frac{\partial^{2}u}{\partial x^{2}%
}=r(t,x)u\text{, }0<x<1,\ t_{0}<t\leq t_{1}. \label{EquationLinear}%
\end{equation}
Then the number of components of%
\[
\{x:0<x<1,\ u\left(  t,x\right)  \not =0\}
\]
is a non-increasing function of $t$.
\end{theorem}

\begin{proof}
We follow similar lines as in \cite[Theorem 6]{Henry85}.

Denote $Q\left(  t\right)  =\{x\in\left(  0,1\right)  :u\left(  t,x\right)
\not =0\}$. We need to show that there is an injective map from the components
of $Q\left(  t_{1}\right)  $ to the components of $Q\left(  t_{0}\right)  $ if
$t_{1}>t_{0}$. If we denote by $C$ a component of $Q\left(  t_{1}\right)  $
and by $S_{C}$ the component of $[t_{0},t_{1}]\times\left(  0,1\right)
\cap\{u\left(  t,x)\not =0\right)  \}$ which contains $C$, then in order to
obtain the injective map it is necessary to prove two facts:

\begin{enumerate}
\item $S_{C}\cap Q(t_{0})\not =\varnothing;$

\item If $C_{1},C_{2}$ are two components of $Q\left(  t_{1}\right)  $, then
$S_{C_{1}}\cap S_{C_{2}}=\varnothing.$
\end{enumerate}

Let us prove the first statement by contradiction, so assume that $S_{C}\cap
Q(t_{0})=\varnothing.$ We can assume without loss of generality that $r\left(
t,x\right)  <0$, because this property is satisfied for the function $W\left(
t,x\right)  =u\left(  t,x\right)  e^{-\lambda t}$ with $\lambda>0$ large
enough and the components of these two functions coincide. Consider for
example that $u\left(  t,x\right)  >0$ in $S_{C}$. Let $M=\max_{\left(
t,x\right)  \in S_{C}}u\left(  t,x\right)  $. By hypothesis and the Dirichlet
boundary conditions this maximum has to be attained at a point $\left(
t^{\prime},x^{\prime}\right)  $ such that $t_{0}<t^{\prime}\leq t_{1}$,
$0<x^{\prime}<1$. Also, there has to be an $\varepsilon>0$ such that if
$\left(  t,x\right)  \in S_{C}$ and $t_{0}<t\leq t_{0}+\varepsilon$, then
$u\left(  t,x\right)  <M$, as otherwise there would be a sequence $\left(
t_{n},x_{n}\right)  \in S_{C}$, $t_{n}>t_{0},$ such that $t_{n}\rightarrow
t_{0}$ and $u\left(  t_{n},x_{n}\right)  =M$. By the continuity of $u$ this
would imply that $u\left(  t_{0},x_{0}\right)  =M$ for some $\left(
t_{0},x_{0}\right)  \in S_{C}$, which is a contradiction. Then we can choose
$t^{\prime}$ as the first time when the maximum is attained, so $u\left(
t,x\right)  <M$ for all $\left(  t,x\right)  \in S_{C},\ t_{0}<t<t^{\prime}$.
By the continuity of $u$ there exists a rectangle $R=[t^{\prime}%
-\delta,t^{\prime}]\times\lbrack x^{\prime}-\gamma,x^{\prime}+\gamma]$ such
that $R$ belongs to $S_{C}$. In order to apply Theorem \ref{Maximum} we put
$\mathcal{O}=R$ and
\[
Q_{\gamma,\delta}=\{(t,x):t\in(t^{\prime}-\delta,t^{\prime}),\left\vert
x-x^{\prime}\right\vert <\gamma\}.
\]
We have that%
\[
\underset{(t,x)\in Q_{\nu\gamma,\sigma_{1}}}{\sup}\ u(t,x)=M,
\]
for some $0<\nu<1$ and any $0<\sigma_{1}<\delta$. Since $u$ satisfies
(\ref{IneqMax}), we conclude from Theorem \ref{Maximum} that $u\left(
t,x\right)  =M$ for all $\left(  t,x\right)  \in Q_{\rho,\sigma}$, which is a contradiction.

For the second statement suppose the existence of two disjoints components
$C_{1},C_{2}$ of $Q\left(  t_{1}\right)  $ such that $S_{C_{1}}\cap S_{C_{2}%
}\not =\varnothing$, which implies in fact that $S_{C_{1}}=S_{C_{2}}$. In this
case we can assume that $r\left(  t,x\right)  >0$, being this justified by the
function $W\left(  t,x\right)  =u\left(  t,x\right)  e^{\lambda t}$ with
$\lambda>0$ large enough. Let for example $u\left(  t,x\right)  >0$ in
$S_{C_{1}}$ and assume that the interval $C_{1}$ has lesser values than the
interval $C_{2}$. Also, it is clear that between $C_{1}$ and $C_{2}$ there
must exist a point $\left(  t_{1},x_{0}\right)  $ such that $u\left(
t_{1},x_{0}\right)  =0$. On the other hand, the set $S_{C_{1}}\cap\left(
t_{0},t_{1}\right)  \times\lbrack0,1]$ is path connected. Thus, there exists a
simple path $\xi$ such that one end point is in $\{t_{1}\}\times C_{1}$ and
the other one is in $\{t_{1}\}\times C_{2}$. Let us consider the set $L$ of
all points which are above the curve $\xi$ and such that the function $u$
vanishes at them. This set is non-empty because $\left(  t_{1},x_{0}\right)
\in L$. Since $L$ is compact, the function $g:L\rightarrow\lbrack t_{0}%
,t_{1}]$ given by $g\left(  t,x\right)  =t$ attains it minimum at a certain
point $\left(  t^{\prime},x^{\prime}\right)  \in L$ such that $t_{0}%
<t^{\prime}$. Then there exists a set $R=[t^{\prime}-\delta,t^{\prime}%
)\times\lbrack x^{\prime}-\gamma,x^{\prime}+\gamma]$ which belongs to
$S_{C_{1}}$. Let $\mathcal{O}=R$ and
\[
Q_{\gamma,\delta}=\{(t,x):t\in(t^{\prime}-\delta,t^{\prime}),\left\vert
x-x^{\prime}\right\vert <\gamma\}.
\]
We have that%
\[
\underset{(t,x)\in Q_{\nu\gamma,\sigma_{1}}}{\inf}\ u(t,x)=0,
\]
for some $0<\nu<1$ and any $0<\sigma_{1}<\delta$. Since $u$ satisfies
(\ref{IneqMin}), we conclude from Theorem \ref{Minimum} that $u\left(
t,x\right)  =0$ for all $\left(  t,x\right)  \in Q_{\rho,\sigma}$, which is a contradiction.
\end{proof}

\bigskip

\textbf{Acknowledgments.}

The first author is a fellow of the FPU program of the Spanish Ministry of
Education, Culture and Sport, reference FPU15/03080.

This work has been partially supported by the Spanish Ministry of Science,
Innovation and Universities, project PGC2018-096540-B-I00, by the Spanish
Ministry of Science and Innovation, project PID2019-108654GB-I00, and by the
Junta de Andaluc\'{\i}a and FEDER, project P18-FR-4509.

During the preparation of this manuscript our colleague and friend Mar\'{\i}a
Jos\'{e} Garrido-Atienza passed away. She was a very kind and warm person and
we will miss her a lot. We dedicate this paper to her memory.

We would like to thank also the referees for their useful comments.

\end{document}